\documentclass[11pt]{amsart}
\usepackage{amssymb, amsmath, amsthm, mathrsfs, amsopn, latexsym}
\usepackage{amsrefs}

\usepackage{graphicx}
\usepackage{color}
\usepackage[final]{hyperref}
\usepackage[a4paper, centering]{geometry}
\DeclareMathOperator{\Tr}{Tr}
\newcommand*{\Trace}[3][\pm]{\Tr^{#1}(#2, #3)}
\newcommand*{\Trp}[2]{\Trace[i]{#1}{#2}}
\newcommand*{\Trm}[2]{\Trace[e]{#1}{#2}}
\newcommand*{\Trie}[2]{\Trace[i,e]{#1}{#2}}
\newcommand*{\chiut}[1][]{\chi^{#1}_{\{u > t\}}}
\newcommand*{\jump}[1]{\Theta_{#1}}

\def\nuint{\widetilde{\nu}}

\newcommand{\A}{\boldsymbol{A}}

\newcommand{\Cyl}[2]{\left[\!\left[ #1 \cdot #2 \right]\!\right]}

\def\DM{{\mathcal{DM}^{\infty}}}
\newcommand*{\BVA}[1][\Omega]{BV({#1}) \cap L^1(#1, |\Div\A|)}

\newcommand*{\BVAlocloc}[1][\Omega]{BV_{\rm loc}({#1})\cap L^1_{\rm loc}(#1, |\Div\A|)}
\newcommand*{\DMlocloc}[1][\Omega]{\mathcal{DM}^{\infty}_{{\rm loc}}{(#1)}}

\newcommand*{\Ixr}[1][x,r]{I^{#1}}

\newcommand{\mean}[1]{\,-\hskip-1.08em\int_{#1}} 
\newcommand{\Leb}[1]{\mathcal{L}^{#1}} 

\def\bbbr{\mathbb{R}}
\def\R{\mathbb{R}}

\def\N{\mathbb{N}}

\def\cinftio{C^{\infty}_c(\Omega)}
\def\intom#1{\int_{\Omega} #1\,  dx}

\DeclareMathOperator{\diver}{div}
\DeclareMathOperator{\Div}{div}

\newcommand{\medint}{-\kern  -,375cm\int}
\newcommand{\medintinrigo}{-\kern  -,315cm\int}

 \newcommand{\hh}{{\mathcal H}^{N-1}}

\newcommand{\LLN}{{\mathcal L}^N}
\newcommand{\LLU}{{\mathcal L}^1}

\newcommand{\Haus}[1]{{\mathcal H}^{#1}} 

\newcommand{\res}{\mathop{\hbox{\vrule height 7pt width .5pt depth 0pt
\vrule height .5pt width 6pt depth 0pt}}\nolimits} 

\def\pscal#1#2{\left\langle #1,\, #2 \right\rangle}

\DeclareMathOperator{\supp}{supp}
\DeclareMathOperator{\Tan}{Tan}

\def\ut{\widetilde{u}}
\def\polar{\theta}

\renewcommand{\prec}[2][*]{#2^{#1}}

\def\intA#1{\int_{\Omega} #1\,  d(\Div\A)}

\def\pair#1{\left(#1\right)}
\def\upiu{u^+}
\def\umeno{u^-}
\def\uint{{u^{i}}}
\def\uext{{u^{e}}}

\def\radon{\mathcal{M}(\Omega)}

\long\def\taglio#1{}

\newtheorem{definition}{Definition}[section]

\newtheorem{lemma}[definition]{Lemma}

\newtheorem{theorem}[definition]{Theorem}

\newtheorem{proposition}[definition]{Proposition}

\newtheorem{corollary}[definition]{Corollary}
\theoremstyle{remark}
\newtheorem{remark}[definition]{Remark}
\newtheorem{example}[definition]{Example}

\makeatletter
\def\@settitle{\begin{center}%
		\baselineskip14\p@\relax
		\bfseries
		\uppercasenonmath\@title
		\@title
		\ifx\@subtitle\@empty\else
		\\[5ex]
		\normalsize\mdseries\@subtitle
		\fi
	\end{center}%
}
\def\subtitle#1{\gdef\@subtitle{#1}}
\def\@subtitle{}
\makeatother
\usepackage{hyperref}

\definecolor{grey}{rgb}{.7,.7,.7}
\definecolor{evidG}{rgb}{0,0.5,0}

\begin{document}
\title[Representation formulas for pairings]{Representation formulas for pairings between divergence-measure fields and $BV$ functions}

\author[G.E.~Comi]{Giovanni E.~Comi}
\address{Dipartimento di Matematica, Università di Bologna, Piazza di Porta San Donato 5, 40126 Bologna (BO), Italy}
\email{giovannieugenio.comi@unibo.it}
\author[G.~Crasta]{Graziano Crasta}
\address{Dipartimento di Matematica ``G.\ Castelnuovo'', Sapienza Università di Roma\\
	P.le A.\ Moro 5 -- I-00185 Roma (Italy)}
\email{graziano.crasta@uniroma1.it}
\author[V.~De Cicco]{Virginia De Cicco}
\address{Dipartimento di Scienze di Base  e Applicate per l'Ingegneria, Sapienza Università di Roma\\
	Via A.\ Scarpa 10 -- I-00185 Roma (Italy)}
\email{virginia.decicco@uniroma1.it}
\author[A.~Malusa]{Annalisa Malusa}
\address{Dipartimento di Matematica ``G.\ Castelnuovo'', Sapienza Università di Roma\\
	P.le A.\ Moro 5 -- I-00185 Roma (Italy)}
\email{annalisa.malusa@uniroma1.it}

\thanks{\textit{Acknowledgments}. 
The authors are members of  the Istituto Nazionale di Alta Matematica (INdAM), Gruppo Nazionale per l'Analisi Matematica, la Probabilità e le loro Applicazioni (GNAMPA), and are partially supported by the INdAM--GNAMPA 2022 Project \textit{Alcuni problemi associati a funzionali integrali: riscoperta strumenti classici e nuovi sviluppi}, codice CUP\_E55F22000270001.
Part of this work was undertaken while the first author was visiting the Mathematics Department Guido Castelnuovo and the SBAI Department at the Sapienza University of Rome. He would like to thank these institutions for the support and warm hospitality during the visits.
}

\date{\today}

\begin{abstract}
The purpose of this paper is to find pointwise representation formulas for 
the density of the pairing between 
divergence-measure fields and BV functions, in this way continuing the research started in \cite{CD3, CDM}. In particular, we extend a representation formula from an unpublished paper of Anzellotti \cite{Anz2} involving the limit of cylindrical averages for normal traces, and we exploit a result of \cite{Silh} in order to derive another representation in terms of limits of averages in half balls.
\end{abstract}

\maketitle


\section{Introduction}
Starting from the pioneering papers \cite{Anz, ChenFrid, Silh}, 
the research community has shown a growing interest in the pairing
theory between divergence-measure fields (i.e. vector fields whose weak divergence is a Radon measure) and $BV$ functions, 
fostered by its application in several contests.
We mention, among others,
hyperbolic conservation and balance laws with discontinuous fluxes
\cite{ChFr1,ChenFrid,ChTo}, 
capillarity and prescribed mean curvature problems \cite{ComiLeo, LeoSar,LeoSar2,MR3501836}, 
the weak formulations of problems involving the 1-Laplacian operator
\cite{MR1814993,MR2348842,AVCM,Cas, MR3813962, MR2502520},
and continuum mechanics \cite{ChCoTo,DGMM,Silh,Schu}. For recent extensions to non-Euclidean and fractional frameworks we refer to \cite{ComiMagna,BuffaComiMira,comi2023fractional}.

The current general setting for the pairing theory is the following (see e.g.\ \cite{Anz, ChenFrid, ComiPayne}). Given an open set $\Omega\subset\R^N$, we say that a vector field $\A \in L^{\infty}(\Omega; \R^N)$ is a divergence-measure field, and we write $\A \in \DM(\Omega)$, if $\Div \A$ is a finite Radon measure on $\Omega$. For any function $u\in BV(\Omega) \cap L^{\infty}(\Omega)$,
the pairing $\pair{\A,Du}$ 
is defined in the sense of distributions as
\begin{equation*}
\pscal{\pair{\A,Du}}{\varphi}:= -\intA{\prec{u}\,\varphi}- 
\intom{u\A\cdot\nabla\varphi}\,,
\qquad
\text{for all } \varphi\in C^\infty_c(\Omega)\,.
\end{equation*}
We recall that this definition is well posed, since the measure $\Div\A$ does not charge sets of $(N-1)$-dimensional
Hausdorff measure zero, while the precise representative $\prec{u}$ of a $BV$ function $u$ 
is defined $\mathcal H^{N-1}$-a.e.\ in $\Omega$. 
In fact, it is proved in \cite{Anz, ChenFrid} that the pairing $\pair{\A,Du}$ is a Radon measure in $\Omega$, 
and $|\pair{\A,Du}| \leq \|\A\|_{L^{\infty}(\Omega; \R^N)} |Du|$, so that there exists a density
$\polar(\A,Du,\cdot) \in L^1(\Omega, |Du|)$ such that
the equality
\begin{equation*}
\pair{\A,Du}=\polar(\A,Du,x)|Du|
\end{equation*}
holds in the sense of measures. In addition, in \cite{CD3, MR3939259, CDM} it is shown that the boundedness assumption on $u$ can be replaced with the weaker requirement that $u^* \in L^1(\Omega; |\Div \A|)$.

\smallskip
The aim of this paper is to find some representation formula for the pairing, obtained through 
a detailed description of the density $\theta(\A, Du, x)$ of the pairing measure $\pair{\A, Du}$ with respect to $|Du|$.

In the classical case, when $\A\in \DM(\Omega) \cap C(\Omega; \R^N)$ and $u\in BV(\Omega) \cap C^1(\Omega)$, then $\pair{\A, Du} = \A \cdot \nabla u\, \LLN$, where $\LLN$ is the $N$-dimensional Lebesgue measure,
so that $\polar(\A,Du,x) =0$ if $\nabla u(x) = 0$, whereas
\[
\polar(\A,Du,x)  = \A(x) \cdot \frac{\nabla u (x)}{|\nabla u(x)|} 
=: \Trace[]{\A}{\{u = u(x)\}}\,,
\qquad \text{if}\ \nabla u(x) \neq 0,
\]
where $\Trace[]{\A}{\{u = u(x)\}}$ denotes the normal trace of $\A$ on the regular level set $\{u = u(x)\}$.

In the general case, 
the usual decomposition 
$$Du = \nabla u \, \LLN + D^j u + D^c u$$ of the measure $Du$ into its absolutely continuous, jump and Cantor parts,
leads to a corresponding decomposition of the pairing measure
$$\pair{\A, Du} = \pair{\A, Du}^a + \pair{\A, Du}^j + \pair{\A, Du}^c.$$

The absolutely continuous part satisfies $\pair{\A, Du}^a = \A \cdot \nabla u \, \Leb{N},$
as it is first shown in \cite{ChenFrid}, while in \cite{CD3} the authors prove that $\pair{\A, Du}^j = \Trace[*]{\A}{J_u}\, |D^j u|$,
where $\Trace[*]{\A}{J_u}$ is the average of the interior and exterior weak normal traces of the vector field $\A$ on 
the jump set $J_u$ of $u$ (see Section~\ref{distrtraces} below).
This result is satisfactory for what concerns the representation of the absolutely continuous part and the jump part of the pairing measure,
since it implies that
\[
\polar(\A,Du,x) = \A \cdot \frac{\nabla u(x)}{|\nabla u(x)|}\, \chi_{\{\nabla u \neq 0\}}(x)\,,
\qquad 
\text{for $\LLN$-a.e.}\ x\in\Omega, 
\]
and
\[
\polar(\A,Du,x) = \Trace[*]{\A}{J_u}(x)\,,
\qquad 
\text{for $|D^j u|$-a.e.}\ x\in\Omega.
\]
On the other hand, for what concerns the representation of the Cantor part 
of the pairing measure, in \cite{CD3} it is proved that
\((\A, Du)^c\res (\Omega\setminus S_{\A}) = 
\widetilde{\A} \cdot D^c u\res (\Omega\setminus S_{\A})\),
where \(S_{\A}\) is the approximate discontinuity set  of \(\A\),
and $\widetilde{\A}$ denotes the approximate continuous representative of $\A$
in $\Omega\setminus S_{\A}$.
Yet, this representation on $\Omega\setminus S_{\A}$ is far from being optimal,
since there exist vector fields $\A \in \DM(\R^N)$
such that the Hausdorff dimension of $S_{\A}$ is $N$ (see Example~\ref{dimensionSA}).

\smallskip
Aiming to obtain a representation formula for $\polar(\A,Du,\cdot)$ without any additional assumption,   
we propose a new approach based on the use of the coarea formula for the pairing measure proved in \cite[Theorem 4.2]{CD3}. 
The basic idea is to use the representation of the purely jump measure $(\A, D\chi_{\{u>t\}})$
on superlevel sets of $u$,
and to recover information on $(\A, Du)$ through the coarea formula, 
obtaining in particular that
\[
(\A, Du)^c = \Trace[*]{\A}{\partial^* \{u>\widetilde u(\cdot)\} }(\cdot)|D^cu|
\]
where $\widetilde{u}$ is the approximate limit of $u$ at some Lebesgue point and $\partial^* E$ denotes the reduced boundary of some measurable set $E$ (see Theorem~\ref{t:repr}).

Hence, more explicit representation formulas for $\polar(\A,Du,\cdot)$ can be inherited by
explicit representation formulas for the weak normal traces of $\A$.

A relevant contribution in this direction is contained in the unpublished paper \cite{Anz2}, 
where the divergence of the vector field $\A\in L^\infty(\Omega;\R^N)$ is assumed to be a summable function, 
and the weak normal trace of $\A$ 
is obtained as the limit of a suitable cylindrical average. 
More precisely,  \cite[Theorem~3.6]{Anz2}  states that, if
$\Div \A \in L^1(\Omega)$
and $u\in BV(\Omega)\cap L^\infty(\Omega)$, then
\begin{equation}\label{f:denscyl1}
\theta(\A, Du, x) =
\Cyl{\A}{\nu_u}(x)
\qquad
\text{for $|Du|$-a.e.}\ x\in\Omega,
\end{equation}
where $Du = \nu_u\, |Du|$ is the polar decomposition of $Du$, 
and, for some set $G \subset \Omega$ and some function $\zeta : G \to \mathbb{S}^{N-1}$,
\begin{equation*}
\Cyl{\A}{\zeta}(x)
:= \lim_{\rho\downarrow 0}
\lim_{r\downarrow 0}
\frac{1}{\LLN(C_{r,\rho}(x, \zeta(x)))}
\int_{C_{r,\rho}(x, \zeta(x))} \A(y) \cdot \zeta(x) \, dy \quad \text{ for } x \in G,
\end{equation*}
whenever the limits exist, with
\[
C_{r,\rho}(x, \zeta(x)) :=
\left\{
y\in\R^N:\ 
|(y-x)\cdot\zeta(x)| < r,\
|(y-x) - [(y-x)\cdot\zeta(x)]\zeta(x)| < \rho
\right\}
\]
(the existence of the limit in \eqref{f:denscyl1} for $|Du|$-a.e. $x \in \Omega$ is part of the statement).

In Theorem~\ref{representAnz}
we obtain a generalization of above-mentioned result
to divergence-measure vector fields, 
by adapting the arguments of the original proof through the use, as a new ingredient, 
of the Gauss-Green formulas recently proved in \cite{CD3} (see Theorem \ref{t:gGG}).
More precisely, the representation formula \eqref{f:denscyl1} turns out to hold true provided that the set on which the jump part of the measure $\Div \A$ is concentrated, which we denote by $\Theta_{\A}$, has an $\Haus{N-1}$-negligible intersection with the jump set $J_u$ of the function $u$ (see also Remark \ref{rem:Cantor_1}).
This result is optimal: if instead $\Haus{N-1}(\Theta_{\A} \cap J_u) > 0$, then relation \eqref{f:denscyl1} is no longer valid, as it is shown in Example~\ref{e:no}.

As an application, we obtain the following Gauss-Green formula,
valid for every $\A\in\DM(\R^N)$ satisfying $\Haus{N-1}(\Theta_{\A})=0$, $u\in BV(\R^N)\cap L^\infty(\R^N)$,
and $E\subset\R^N$ set of finite perimeter such that $\supp(\chi_E u)$ is bounded:
\[
\int_{E^1} u^* \, d\Div\A + \int_{E^1} \Cyl{\A}{\nu_u} \, d|Du| = -
\int_{\partial ^*E} u^i \ \Cyl{\A}{\nu_{E}} \,
d\mathcal H^{N-1}\,,
\]
where $E^1$ denotes the measure theoretic interior of $E$,
and $u^i$ denotes the interior trace of $u$ on $\partial^* E$ 
(see Theorem~\ref{t:GG}).

In addition, in the paper \cite{CD5} the cylindrical averages approach is further exploited in order to gain an explicit representation for the relaxation of a pairing-type functional.

\smallskip
Finally, combining our results
with the representation formula for weak normal traces obtained in \cite[Theorem~4.4]{Silh},
we get the following further integral representation
\begin{equation*}
\theta(\A, Du, x) =
\lim_{r\to 0}
\frac{N}{2 \omega_{N-1} r^N}
\left(
\int_{B_r^i(x)} \A(y) \cdot \frac{y-x}{|y-x|}\, dy
-
\int_{B_r^e(x)} \A(y) \cdot \frac{y-x}{|y-x|}\, dy
\right)\,,
\end{equation*}
for $|Du|$-a.e.\ $x\in\Omega$, where \(B_r^i(x) := \{y\in B_r(x):\ (y-x)\cdot \nu_u(x) > 0\}\), 
\(B_r^e(x) := \{y\in B_r(x):\ (y-x)\cdot \nu_u(x) < 0\}\),
and $\omega_{N-1}$ is the $(N-1)$-dimensional Lebesgue measure of
the unit ball in $\R^{N-1}$.

As a further application of our general representation formula,
in the final part of the paper
we recover the local structure of the pairing measure by means of its
tangent measures, coherently with the classical theory of sets of finite perimeter and functions of bounded variation. This result can be also achieved by means of direct calculations on the blow-up sequence of the pairing measure; however this is not necessary, since we can exploit Theorem~\ref{t:repr} below and the Federer-Vol'pert theorem.

\medskip
The plan of the paper is the following.

In Section~2 we set the notation, and we recall some results on divergence-measure vector fields, their
weak normal traces and functions of bounded variation.

In Section~3 we first recall some results concerning the pairing between divergence-measure vector fields
and functions of bounded variation, mainly taken from
\cite{Anz,ChenFrid,CD3,CDM}.
Then we prove the result on the representation of the density
$\polar(\A,Du,\cdot)$ in terms of weak normal traces (Theorem~\ref{t:repr}).

In Section~4, building on the results of Section~3, we show that $\polar(\A,Du,\cdot)$ can be represented in terms of the cylindrical averages introduced in \cite{Anz2}; then we achieve a similar result with the half balls averages introduced in \cite{Silh}.

Finally, in Section~5 we briefly describe the tangential properties of the pairing measure.


\section{Notation and preliminary results}
\label{s:prelim}

In the following we denote by \(\Omega\) a nonempty open subset of 
\(\R^N\), and for every set $E\subset \R^N$ we denote by $\chi_{E}$ its 
characteristic function. We say that a set $E$ is compactly contained in $\Omega$, and we write $E \Subset \Omega$, if the closure \(\overline{E}\) of \(E\) is a compact subset of \(\Omega\). Given two sets $E, F \subset \R^N$, their symmetric difference is the set $E\vartriangle F := (E \setminus F) \cup (F \setminus E)$.
For $x \in \R^N$ and $r > 0$, we denote by $B_{r}(x)$ the ball centered in $x$ with radius $r$, and we set $B_1 := B_1(0)$.

\subsection{Measures}

We denote by
$\LLN$ 
and $\hh$
the Lebesgue measure 
and the $(N-1)$-dimensional 
Hausdorff measure in $\R^N$, respectively. Unless otherwise stated, a measurable set is a $\Leb{N}$-measurable set. We set $\omega_N := \Leb{N}(B_1)$. 

Following the notation of \cite{AFP}, we denote by $\mathcal{M}_{\rm loc}(\Omega)$ the space of Radon measures on $\Omega$, and by $\radon$ the space of finite Radon measures on $\Omega$.

Given $\mu\in \mathcal{M}(\Omega)$ and a $\mu$-measurable set $E$, the {\sl restriction} $\mu\res E$ is the Radon measure defined by
\[
\mu\res E(B):=\mu(E\cap B), \qquad \forall\ B\ \text{$\mu$-measurable},\ 
B\subset\Omega.
\]

The {\sl total variation} $|\mu|$ of $\mu \in \mathcal{M}(\Omega)$ is 
the nonnegative Radon measure defined by
\[
|\mu|(E) := \sup\left\{ \sum_{h=0}^\infty |\mu(E_h)| \colon \ E_h\ 
\text{$\mu$-measurable sets, pairwise disjoint},\ E=\bigcup_{h=0}^\infty E_h 
\right\},
\]
for every $\mu$-measurable set $E$.
Since $\mu \in \mathcal{M}_{\rm loc}(\Omega)$ if and only if $\mu \in \mathcal{M}(\Omega')$ for every open set $\Omega' \Subset \Omega$, the above definitions can be easily extended to the case of a not necessarily finite Radon measure $\mu$ by adding the assumptions $B \Subset \Omega$ and $E_h \Subset \Omega$, respectively.

\medskip  

Let $\mu \in \mathcal{M}_{\rm loc}(\Omega)$ and $\nu$ be a nonnegative measure on $\Omega$. We say that $\mu$  is {\sl absolutely continuous} with respect to $\nu$ (notation: $\mu \ll \nu$) 
if $|\mu|(B)=0$ for every set $B \subset \Omega$ such that $\nu(B)=0$. If, in addition, $\nu$
is $\sigma$-finite, then, by the Radon-Nikod\'ym Theorem, there exists a unique 
function $\theta\in L^1_{\rm loc}(\Omega, \nu)$ (called {\sl density} of $\mu$ w.r.t. 
$\nu$) such that $\mu=\theta \nu$ (clearly, $\theta\in L^1(\Omega, \nu)$ if and only if $|\mu|(\Omega) < \infty$).

In the special case of $\nu= |\mu|$, the density
$\theta$ satisfies $|\theta|=1$ $\mu$-a.e.\ in $\Omega$, and 
$\mu=\theta |\mu|$ is called {\sl polar decomposition} of $\mu$.

Two positive measures $\nu_1$, $\nu_2\in \mathcal{M}_{\rm loc}(\Omega)$ are {\sl mutually 
singular}
(notation: $\nu_1 \perp \nu_2$) 
if there exists a Borel set $E \subset \Omega$ such that $|\nu_1|(E)=0$ and 
$|\nu_2|(\Omega\setminus E) = 0$.

\medskip

Given a nonnegative measure $\mu\in\mathcal{M}_{\rm loc}(\Omega)$ and a  
function $f\in L^1_{\rm loc}(\Omega,\mu)$, we say that $f$ has an {\sl approximate limit} \(z\in\R\) at $x\in\Omega$ if
\begin{equation}\label{f:leb0}
\lim_{r\rightarrow0^{+}}\frac{1}{\mu\left(  B_r(x)\right)}\int_{B_r\left(  
x\right)
}\left|  f(y)  - z  \right|  \,d\mu(y)=0.
\end{equation}
In this case, we say that $x$ is a {\em Lebesgue point} of $f$ with respect to $\mu$. Thanks to Lebesgue's differentiation theorem, we know that $\mu$-almost every $x \in \Omega$ is a Lebesgue point of $f$ with respect to $\mu$. In addition, in every Lebesgue point of $f$ with respect to $\mu$ the approximate limit is uniquely determined and is denoted by $z := \widetilde{f}(x)$. In what follows we choose $\widetilde{f}$ as pointwise representative of $f\in L^1_{\rm loc}(\Omega,\mu)$; that is, we assume $f(x) := \widetilde f(x)$ in every Lebesgue point and whenever this choice does not cause any ambiguity.

\subsection{Divergence-measure fields }
\label{ss:div}

We denote by \(\DM(\Omega)\) the space of all
vector fields 
\(\A\in L^\infty(\Omega; \R^N)\)
whose divergence in the sense of distributions is a finite Radon measure in 
\(\Omega\), acting as
\[
\int_\Omega \varphi\, d\Div\A = -\int_{\Omega} \A\cdot\nabla\varphi\, dx \qquad
\forall \varphi\in\cinftio.
\]
Similarly, \(\DMlocloc[\Omega]\) will denote the space of
all vector fields \(\A\in L^\infty_{{\rm loc}}(\Omega; \R^N)\)
whose divergence in the sense of distributions is a Radon measure in 
\(\Omega\). 

The main property, proved in \cite[Proposition 3.1]{ChenFrid} (see also \cite[Theorem 3.2]{Silh}), is that for 
every \(\A \in \DMlocloc[\Omega]\) the measure $\Div\A$ is absolutely 
continuous with respect to the Hausdorff measure $\hh$, so that the following 
decomposition result holds, which is the localized version of \cite[Proposition 2.3]{CDM} (see also \cite[Proposition~2.3]{ADM}).

\begin{proposition}\label{p:basicVF}
Given  a vector field \(\A \in \DMlocloc[\Omega]\), the set
\begin{equation}\label{f:jump}
\jump{\A} 
:= \left\{
x\in\Omega:\
\limsup_{r \to 0^+}
\frac{|\Div \A| (B_r(x))}{r^{N-1}} > 0
\right\}
\end{equation} 
is a Borel set, \(\sigma\)-finite with respect to  \(\hh\), and the measure
\(\Div\A\) can be decomposed as the sum of mutually singular measures \(
\Div\A = \Div^a\A + \Div^c\A + \Div^j\A,\) 
where 
\begin{itemize} 
\item[(i)] $\Div^a\A \ll \LLN$;
\item[(ii)] \(\Div^c\A (B) = 0\) for every Borel set \(B\) with \(\hh(B) < +\infty\);
\item[(iii)] there exists $f\in L^1_{\rm loc}(\jump{\A}, \Haus{N-1}\res \jump{\A})$
such that
\(
\Div^j\A = f\, \hh\res\jump{\A}
\), which implies $\Div^j \A \ll \hh\res\jump{\A}$.
\end{itemize}
\end{proposition}

In what follows, we will call \(\jump{\A}\) the {\em jump set} of the measure \( \Div\A\).

\subsection{Weak normal traces on oriented countably $\Haus{N-1}$-rectifiable sets} \label{distrtraces}
We recall that $\Sigma \subseteq \R^N$ is a {\sl countably $\Haus{N-1}$-rectifiable set} if there exist (at most) countably many $C^1$ embedded hypersurfaces \((\Sigma_k)_{k\in\N} \subseteq \R^N\) such that 
$
\hh(\Sigma\setminus \bigcup_k \Sigma_k) = 0.
$
A notion of orientation on rectifiable sets can be given as follows:
if we choose {\sl oriented} hypersurfaces \((\Sigma_k)\), 
we define $\hh$-a.e.\ on $\Sigma$ an orientation $\nu_\Sigma$ by selecting 
pairwise disjoint Borel sets \(N_k\subseteq \Sigma_k\) 
such that 
$
\hh(\Sigma\setminus \bigcup_k N_k) = 0
$
and by setting
$
\nu_\Sigma=\nu_{\Sigma_k} \ \text{ on }\ N_k. 
$ 
This orientation depends clearly on the choice of the decomposition, but only 
up to the sign, due to the fact that for any pair of $C^1$ hypersurfaces 
$\Gamma$, $\Gamma'$ it holds that $\nu_{\Gamma'}\in \{-\nu_{\Gamma}, 
\nu_{\Gamma}\}$ $\hh$-a.e. on $\Gamma\cap\Gamma'$.

\medskip

In what follows, we
will deal with the traces of the normal component of a vector field \(\A\in 
\DMlocloc\) on an oriented countably \(\Haus{N-1}\)-rectifiable set
\(\Sigma\subset\Omega\).
In order to fix the notation, we briefly recall the construction given in 
\cite[Propositions~3.2, 3.4 and Definition~3.3]{AmbCriMan}.

Given a domain \(\Omega'\Subset\Omega\) of class \(C^1\), 
the trace of the normal component of \(\A\) on \(\partial\Omega'\) 
is the distribution defined by
\begin{equation} \label{f:deftrace}
\pscal{\Trace[]{\A}{\partial\Omega'}}{\varphi}
:= - \int_{\Omega'} \A\cdot \nabla\varphi\, dx - \int_{\Omega'} \varphi\, d\Div\A,
\qquad
\forall\varphi\in C^\infty_c(\Omega).
\end{equation}
It turns out that this distribution is induced by an \(L^\infty\) function on 
\(\partial\Omega'\),
still denoted by \(\Trace[]{\A}{\partial\Omega'}\), and
\begin{equation}\label{f:infestrace}
\|\Trace[]{\A}{\partial\Omega'}\|_{L^\infty(\partial\Omega', \Haus{N-1}\res 
\partial\Omega')}
\leq \|\A\|_{L^\infty(\Omega'; \R^N)}.
\end{equation}

Given an oriented countably $\hh$-rectifiable set \(\Sigma\), and using the 
notation for the covering of $\Sigma$ introduced at the beginning of this 
section, one can prove that
for every $k\in\N$, there exist two open bounded sets \(\Omega_k, 
\Omega'_k\) with 
\(C^1\) boundary
and interior normal vectors \(\nu_{\Omega_k}\) and \(\nu_{\Omega_k'}\), 
respectively,
such that
\(N_k\subseteq \partial\Omega_k \cap \partial\Omega'_k\),
and
\[
\nu_{\Sigma_k}(x) = \nu_{\Omega_k}(x) = -\nu_{\Omega'_k}(x)
\qquad \forall x\in N_k.
\]

By a deep localization property proved in \cite[Proposition 3.2]{AmbCriMan},
we can fix an orientation on \(\Sigma\), given by
\[
\nu_{\Sigma}(x) := \nu_{\Sigma_k}(x), \qquad \hh\text{-a.e. on}\  N_k
\]
and the {\em interior and exterior normal traces} of \(\A\) on \(\Sigma\) 
are defined by
\[
\Trp{\A}{\Sigma} := \Tr(\A, \partial\Omega_k),
\quad
\Trm{\A}{\Sigma} := -\Tr(\A, \partial\Omega'_k),
\qquad
\hh\text{-a.e.\ on}\ N_k,
\]
respectively.

As a consequence, if we consider two oriented countably $\hh$-rectifiable sets $\Sigma$ and $\Sigma'$ with the same orientation and such that $\Haus{N-1}(\Sigma \cap \Sigma') > 0$, then
\begin{equation}\label{locality_traces}
\Trp{\A}{\Sigma} = \Trp{\A}{\Sigma'}
\quad
\Trm{\A}{\Sigma} = \Trm{\A}{\Sigma'}
\qquad
\hh\text{-a.e.\ on}\ \Sigma \cap \Sigma',
\end{equation}
see for instance \cite[Proposition 4.10]{ComiPayne}.

Moreover, the normal traces belong to
\(L^{\infty}(\Sigma, \Haus{N-1}\res\Sigma)\) 
and satisfy
\begin{equation} \label{improved_bound}
\max \{\|\Trp{\A}{\Sigma}\|_{L^{\infty}(\Sigma, \Haus{N-1} \res \Sigma)}, \|\Trm{\A}{\Sigma}\|_{L^{\infty}(\Sigma, \Haus{N-1} \res \Sigma)} \} \le \|\A\|_{L^{\infty}(\Omega'; \R^N)}
\end{equation}
for every open set $\Omega'$ such that $\Sigma \subset \Omega' \Subset \Omega$ (see for instance \cite[Theorem 4.2]{ComiPayne}), and
\begin{equation}\label{f:trA}
\Div \A \res\Sigma =
\left[\Trp{\A}{\Sigma} - \Trm{\A}{\Sigma}\right]
\, {\mathcal H}^{N-1} \res\Sigma
\end{equation}
(see \cite[Proposition~3.4]{AmbCriMan}).
In particular, by \eqref{f:infestrace}, we get
$|\Div\A|(\Sigma) \leq 2\|\A\|_{L^{\infty}(\Omega; \R^N)} \Haus{N-1}(\Sigma)$.

In what follows we use the notation
\[
\Trace[*]{\A}{\Sigma}:= \frac{\Trp{\A}{\Sigma}+\Trm{\A}{\Sigma}}{2}\,.
\]

\begin{remark}\label{r:diffnt}
We stress the fact that in \eqref{f:deftrace} we are using the opposite sign with respect to the definition of normal trace given in \cite{AmbCriMan, Anz}, and so the opposite orientation of the rectifiable hypersurfaces.
Anyway, if $\Sigma$ is oriented by a normal vector field
$\nu$ and $\Sigma'$ is the same set oriented by $\nu' := -\nu$,
then
\[
\Trm{\A}{\Sigma'} = -\Trp{\A}{\Sigma},
\quad
\Trp{\A}{\Sigma'} := -\Trm{\A}{\Sigma},
\]
so that the difference
$\Trp{\A}{\Sigma} - \Trm{\A}{\Sigma}$
in \eqref{f:trA}
is independent of the choice of the orientation on $\Sigma$.
\end{remark}

\subsection{Functions of bounded variation}
\label{ss:BV}

Even if we mostly follow the terminology of \cite{AFP}, nevertheless we recall the 
main conventions and results for reader's convenience.
 
A function \(u\in L^1(\Omega)\) has \textsl{bounded variation} in 
\(\Omega\), and we write $u\in BV(\Omega)$,
if the distributional derivative \(Du\) of \(u\) is a vector valued finite 
Radon measure in 
\(\Omega\).
We denote by \(BV_{{\rm loc}}(\Omega)\) the set of functions
\(u\in L^1_{{\rm loc}}(\Omega)\) that belongs to 
\(BV(\Omega')\) for every open set \(\Omega'\Subset\Omega\). In addition, we let $BV(\Omega; \R^m)$ be the space of $\R^m$-vector valued functions of bounded variations in $\Omega$, and we define analogously the local space $BV_{\rm loc}(\Omega; \R^m)$.

\medskip

In spite of the fact that a $BV$ function $u$ is an $L^1$ function, it admits a representative well defined outside an $\Haus{N-1}$-negligible set. In order to define it, we recall some more results on approximate limits of summable functions.

If in the definition of approximate limit \eqref{f:leb0} we have $\mu = \Leb{N}$ and $f = u\in L^1_{{\rm loc}}(\Omega; \R^m)$, then we say that $x$ is a Lebesgue point of $u$, omitting the reference to the Lebesgue measure. In order to emphasize the distinction with the approximate jump points of $u$ defined below, we use here and in similar situations the notation $\widetilde{u}(x)$ for the pointwise representative of $u$ in its Lebesgue points.
The set $C_u\subset\Omega$ of points where this property holds is called the
\textsl{approximate continuity set} of $u$, whereas
the set \(S_u := \Omega\setminus C_u\)
is called the \textsl{approximate discontinuity set} of $u$.

We say that \(x\in\Omega\) is an {\sl approximate jump point} of \(u\) if
there exist \(a,b\in\R^m\), \(a\neq b\), and a unit vector \(\nu\in\R^N\) such that 
and
\begin{equation}\label{f:disc}
\begin{gathered}
\lim_{r \to 0^+} \frac{1}{\LLN(B_r^i(x))}
\int_{B_r^i(x)} |u(y) - a|\, dy = 0,
\\
\lim_{r \to 0^+} \frac{1}{\LLN(B_r^e(x))}
\int_{B_r^e(x)} |u(y) - b|\, dy = 0,
\end{gathered}
\end{equation}
where \(B_r^i(x) := \{y\in B_r(x):\ (y-x)\cdot \nu > 0\}\), and 
\(B_r^e(x) := \{y\in B_r(x):\ (y-x)\cdot \nu < 0\}\).
The triplet \((a,b,\nu)\), uniquely determined by \eqref{f:disc} 
up to a permutation
of \((a,b)\) and a change of sign of \(\nu\),
is denoted by \((\uint(x), \uext(x), \nu_u(x))\).
The set of approximate jump points of \(u\) is denoted by \(J_u\), and it is clear that $J_u \subset S_u$.

If \(u\in BV_{\rm loc}(\Omega; \R^m)\), then both $J_u$ and $S_u$ are countably 
$\Haus{N-1}$-rectifiable sets, we have $\Haus{N-1}(S_u \setminus J_u)=0$, and 
for $\Haus{N-1}$-a.e.\ $x\in J_u$ the unit vector $\nu_u(x)$ can be identified 
with the normal 
vector $\nu_{J_u}(x)$ defined in Section \ref{distrtraces} for general countably 
$\Haus{N-1}$-rectifiable sets (up to a change in orientation).

\begin{definition}
The precise representative $u^*$ of \(u\in BV_{\rm loc}(\Omega, \R^m)\) is 
defined in $\Omega\setminus(S_u\setminus J_u)$
(hence $\hh$-a.e.\ in $\Omega$) as
\[
u^*(x) :=
\begin{cases}
\ut(x) & \text{ for } x \in \Omega \setminus S_u, \\[4pt]
\dfrac{\uint(x)+\uext(x)}{2} & \text{ for } x \in J_u.
\end{cases}
\]
\end{definition}

In addition, in \cite{CDM} the authors consider a family of representatives of a $BV$ function depending on a Borel map $\lambda : \Omega \to [0, 1]$, namely the {\sl 
$\lambda$--representative} of $u\in BV_{\rm loc}(\Omega; \R^m)$, given by
\begin{equation}\label{f:pr}
u^\lambda(x):=
\begin{cases}
\tilde{u}(x) & \text{ for } x\in \Omega \setminus S_u, \\
(1-\lambda(x))\umeno(x)+\lambda(x)\upiu(x) & \text{ for } x\in J_u.
\end{cases}
\end{equation}
Clearly, for $\lambda(x) = 1/2$ for every $x\in\Omega$, we get $u^\lambda = u^*$.

In the remaining part of this section, we focus on the scalar case $m = 1$.
The gradient measure \(Du\) of a function \(u\in BV(\Omega)\) can be decomposed 
as
the sum of mutually singular measures
$$Du = D^a u + D^j u + D^c u,$$
where $D^au$ is the {\em absolutely continuous part} with respect to the Lebesgue 
measure, that is, 
$D^a u=\nabla u\,  \LLN$ ($\nabla u\in L^1(\Omega; \R^N)$ is the {\em approximate differential} of $u$), while $ D^j u$ is the {\sl jump part}, characterized by
$D^j u =(\prec[i]{u} - \prec[e]{u}) \, \nu_u\, \Haus{N-1}\res J_u$,
and $D^c u$ is the {\sl Cantor part}.
We denote by 
\begin{equation*}
D^d u := D^a u + D^c u
\end{equation*}
the {\em diffuse part} of $Du$, which is concentrated on $C_u$, since $S_u$ is countably $\Haus{N-1}$-rectifiable.

Based on the notion of density of a measurable set $E$ at a point $x\in\R^N$:
\[
D(E;x) := \lim_{\rho\to 0^+} \frac{\LLN(E\cap B_\rho(x))}{\LLN(B_\rho(x))},
\] 
(whenever the limit exists), we define the {\em measure theoretic interior} and {\em exterior} of $E$:
\begin{equation*}
E^{1} := \{ x \in \R^N : D(E; x) = 1 \} \text{ and } E^{0} := \{ x \in \R^N : D(E; x) = 0 \},
\end{equation*}
as well as the {\em measure theoretic boundary} 
\begin{equation} \label{f:MTB}
\partial^M E := \R^N \setminus (E^1 \cup E^0).
\end{equation}
In the case of a general measurable function $u: \Omega \to \R$, we set
\begin{equation*}
\{u \lessgtr t\} := \{ x \in \Omega : u(x) \lessgtr t \}  \text{ for all } t \in \R,
\end{equation*}
and
\[
u^-(x) :=
\sup\left\{t\in\R\colon
D(\{u < t\};x) = 0\right\},
\quad
u^+(x) :=
\inf\left\{t\in\R\colon
D(\{u > t\};x) = 0\right\},
\]
for which, in the case $u \in BV_{\rm loc}(\Omega)$, we have that
\[
\umeno(x) =\min\{\uint(x),\uext(x)\}, \qquad 
\upiu(x) =\max\{\uint(x),\uext(x)\},
\qquad \text{for $\Haus{N-1}$-a.e.}\ x\in J_u,
\]
so that
we can always choose an orientation on $J_u$
such that $\uint = \upiu$ on $J_u$
(see \cite[\S 4.1.4, Theorem~2]{GMS1}).
In what follows we shall always fix this orientation.
Under this assumption, $\nu_u(x)$ coincides, for $\hh$-a.e.\ $x\in J_u$, with
the density in the polar decomposition of the measure $D^j u$
(see \cite[\S4.1.4, Cor.~2]{GMS1}), so that we may write $D^j u = \nu_u \, |D^j u|$.
Due to this identity, with a little abuse of notation, we also denote by $\nu_u$ the density of $Du$ with respect to $|Du|$, so that
$$Du = \nu_u |D u|$$
and $|\nu_u(x)| = 1$ for $|Du|$-a.e. $x \in \Omega$ (thanks to Radon-Nikod\'ym Theorem).

\medskip

A measurable set $E$ is of (locally) finite perimeter in $\Omega$ if 
its characteristic function $\chi_E$ belongs to $BV(\Omega)$ (respectively, $BV_{\rm loc}(\Omega)$).
If $E$ has locally finite perimeter in \(\Omega\),
we call \textsl{reduced boundary} \(\partial^* E\) of \(E\) the set of all 
points
\(x\in \Omega\) in the support of \(|D\chi_E|\) such that the limit
\[
\nuint_E(x) := \lim_{\rho\to 0^+} 
\frac{D\chi_E(B_\rho(x))}{|D\chi_E|(B_\rho(x))}
\]
exists in \(\R^N\) and satisfies \(|\nuint_E(x)| = 1\).
The function \(\nuint_E\colon\partial^* E\to \mathbb{S}^{N-1}\) is called
the \textsl{measure theoretic unit interior normal} to \(E\), and it is clear that $\nuint_E = \nu_{\chi_E}$.

A fundamental result of De Giorgi (see \cite[Theorem~3.59]{AFP}) states that
\(\partial^* E\) is a countably \((N-1)\)-rectifiable set,
\(|D\chi_E| = \hh\res \partial^* E\),
and $\nuint_E(x) = \nu_{\partial^* E}(x)$ for $\Haus{N-1}$-a.e. $x \in \partial^* E$, where $\nu_{\partial^* E}$ is the normal vector to $\partial^* E$, in the sense of Section \ref{distrtraces}. Due to these facts, with a little abuse of notation, we shall simply write $\nu_E$ to denote the measure theoretic unit interior normal, coherently with most of the literature.

If $u\in BV_{\rm loc}(\Omega)$, then the level sets
$E_t := \{u > t\}$ are of locally finite perimeter for $\Leb{1}$-a.e.\ $t\in\R$,
and we have
$\nu_{E_t}(x) = \nu_{\Sigma_{t}}(x) = \nu_u(x)$
for $\hh$-a.e.\ $x\in\Sigma_t$,
where $\Sigma_t := \partial^*\{u>t\}$. In addition, the measure $Du$ can be disintegrated on the level sets of $u$ thanks to the coarea formula (see \cite[Theorem 4.5.9]{FED}).

\begin{theorem}[Coarea formula]\label{coarea}
If $u\in BV_{\rm loc}(\Omega)$, then for $\LLU$-a.e. $t\in\R$ the set $\{u>t\}$ has 
finite 
perimeter in $\Omega$ and  
\begin{equation*}
\int_\Omega
g\,d|Du|=\int_{-\infty}^{+\infty} \int_{\partial^* \{u>t\}
\cap\Omega}\!g\,d\hh\!\, dt
=\int_{-\infty}^{+\infty}
\int_{\{u^-\leq t\leq u^+ \}}\!g\,d\hh\, dt,
\end{equation*}
for every Borel function $g:\Omega\to[0,+\infty]$.
\end{theorem}

Thanks to Theorem~\ref{coarea} and the inclusion
\begin{equation*}
\partial^*{\{u>t\}} \subset \{u^-\leq t\leq u^+ \} \quad \text{ for every } 
t\in\bbbr,
\end{equation*}
we deduce that
\begin{equation*}
\qquad\qquad \hh\Bigl(\{u^-\leq t\leq u^+ \}\setminus
\bigl(\partial^*\{u>t\}\bigr)\Bigr)=0
\qquad \hbox{\rm for $\LLU$-a.e. $t\in\bbbr.$}
\end{equation*}

Specializing the coarea formula to the approximate continuity set $C_u$,
and using the inclusion
\begin{equation}\label{f:incl1}
\partial^* \{u > t\} \cap C_u
\subseteq \{x\in C_u\colon \ut(x) = t\}\,,
\end{equation}
we also get
\begin{equation}\label{f:incl2}
\qquad\qquad \hh\Bigl(\{x\in C_u\colon \ut(x)=t \}\setminus
\bigl(C_u\cap\partial^*\{u>t\}\bigr)\Bigr)=0
\qquad \hbox{\rm for $\LLU$-a.e. $t\in\bbbr.$}
\end{equation}


\section{The pairing measure and its representation}

In order to give the notion of pairing between divergence-measure fields and $BV$ functions, we need a particular subset of the $BV$ space, previously introduced in \cite{CDM}.

\begin{definition}\label{d:BVA}
Given \(\A\in\mathcal{DM}^{\infty}_{\rm loc}(\Omega)\), we define:
\begin{gather*}
\BVA := \left\{
u\in BV(\Omega):\ u^* \in L^1(\Omega, |\Div\A|)
\right\}\,,
\\
\BVAlocloc := \left\{
u\in BV_{\rm loc}(\Omega):\ u^* \in L^1_{\rm loc}(\Omega, |\Div\A|)
\right\}\,.
\end{gather*}
\end{definition}
We remark that $|\Div\A| \ll \Haus{N-1}$ and $u^*$ is defined
$\Haus{N-1}$-a.e.\ in $\Omega$,
hence these definitions are well-posed.

We introduce now the general notion of pairing between a divergence-measure field and a suitable $BV$ function (see \cite[Section 2.5 and Theorem 4.12]{CD3}).

\begin{definition}[Pairing]\label{d:pair}
The pairing between a vector field $\A\in \mathcal{DM}^{\infty}_{\rm loc}(\Omega)$
and a function $u\in\BVAlocloc$
is the distribution
$\pair{\A,Du}\colon C^\infty_c(\Omega) \to \R$ 
acting as
\begin{equation*}
\pscal{\pair{\A,Du}}{\varphi}:= -\intA{\prec{u}\,\varphi}- 
\intom{u\A\cdot\nabla\varphi}\,,
\quad \text{ for }
\varphi\in C^\infty_c(\Omega)\,.
\end{equation*}
\end{definition}

\begin{remark}
Thanks to \cite[Lemma 3.2]{CDM}, we know that, for all Borel function $\lambda_1, \lambda_2 : \Omega \to [0, 1]$, we have \(
\prec[\lambda_1]{u}\in L^1_{\rm{loc}}(\Omega,|\Div \A|)
\)
if and only if
\(
\prec[\lambda_2]{u}\in L^1_{\rm{loc}}(\Omega,|\Div \A|)
\). Therefore one could replace $u^*$ with $u^{\lambda}$ in the definitions of $\BVA$ and $\BVAlocloc$, and obtain the same spaces.
This observation allows the authors of \cite{CDM} to give a more general definition of pairing involving, instead of $u^*$, the $\lambda$-representative $u^{\lambda}$ given by \eqref{f:pr}: more precisely, they define the $\lambda$-pairing $\pair{\A,Du}_\lambda$, acting as
\begin{equation} \label{def:lambda_pair}
\pscal{(\A,Du)_\lambda}{\varphi}:= -\intA{u^\lambda\,\varphi}- 
\intom{u\A\cdot\nabla\varphi}\,,
\quad \text{ for }
\varphi\in C^\infty_c(\Omega)\,.
\end{equation}
Since $(\A, Du)$ and $(\A, Du)_{\lambda}$ differ only on $\Theta_{\A} \cap J_u$, by \cite[Proposition 4.4]{CDM}, we are going to state our results for the standard pairing $(\A, Du)$, underlining possible differences only whenever they appear.\footnote{We point out that in \cite{CDM} the authors denote by $(\A, Du)_*$ the standard pairing (for $\lambda \equiv 1/2$), whereas we use the classical notation.}
\end{remark}

The relevant properties of the pairings are recalled in the following proposition, which is the combination of \cite[Theorem 4.12]{CD3} and \cite[Proposition 2.2 and Corollary 2.3]{MR3939259}.

\begin{proposition}\label{p:pairing0}
Let $\A\in \mathcal{DM}^{\infty}_{\rm loc}(\Omega)$ and $u\in\BVAlocloc$. Then $\pair{\A,Du}$ is a Radon measure in $\Omega$, and 
the equation
\begin{equation*}
\Div(u\A)=\prec{u}\Div\A+\pair{\A,Du}
\end{equation*}
holds in the sense of Radon measures in $\Omega$.
Moreover, $\pair{\A,Du}$ is absolutely continuous
with respect to $|Du|$, with
\begin{equation} \label{eq:pairing_tot_var_bound}
|\pair{\A,Du}| \res \Omega' \leq \|\A\|_{L^{\infty}(\Omega'; \R^N)} |Du| \res \Omega'
\end{equation}  
for every open set $\Omega' \Subset \Omega$.
\end{proposition}

In what follows we will write
\begin{equation*}
\pair{\A,Du}=\polar(\A,Du,x)|Du|,
\end{equation*}
where $\polar(\A,Du,\cdot)$ denotes the Radon--Nikod\'ym derivative of
$\pair{\A,Du}$ with respect to $|Du|$, and our aim is to represent $\polar(\A,Du,\cdot)$ in terms of the weak normal traces of the field $\A$ on the level sets of $u$.

We recall a remarkable decomposition result for the pairing measure, \cite[Theorem 4.12]{CD3}.

\begin{theorem} \label{t:pairing}
Let \(\A\in\mathcal{DM}^{\infty}_{\rm loc}(\Omega)\) and $u\in\BVAlocloc$.
Then the decomposition of the pairing measure into its absolutely continuous, Cantor and jump parts, 
$$(\A, Du) = (\A, Du)^a + (\A, Du)^c + (\A, Du)^j,$$ 
satisfies the following properties:
\begin{itemize}
	\item[(i)]
	absolutely continuous part: 
	\((\A, Du)^a = \A \cdot \nabla u\, \LLN\);
	
	\item[(ii)]
	jump part:
	\(\displaystyle
	(\A, Du)^j = \Trace[*]{\A}{J_u} |D^j u| =
\Trace[*]{\A}{J_u}
	\, (u^+-u^-) \, \hh \res J_u
	\);

	\item[(iii)]
	Cantor part:
	\((\A, Du)^c\res(\Omega\setminus S_{\A}) = 
	\widetilde{\A} \cdot D^c u \res(\Omega\setminus S_{\A})\).
\end{itemize}
\end{theorem}

In addition, we denote by $(\A, Du)^d$ the {\em diffuse part} of the pairing measure; that is,
$$ (\A, Du)^d = (\A, Du)^a + (\A, Du)^c.$$

\begin{remark}
Proposition \ref{p:pairing0} can be seen as the particular case $\lambda \equiv 1/2$ of \cite[Proposition 4.4]{CDM}, which applies to the general $\lambda$-pairing given by \eqref{def:lambda_pair} and provides an estimate analogous to \eqref{eq:pairing_tot_var_bound}: given any Borel function $\lambda : \Omega \to [0, 1]$, we have
\begin{equation*}
|(\A, Du)_{\lambda}| \res \Omega' \leq \|\A\|_{L^{\infty}(\Omega'; \R^N)} |Du| \res \Omega'
\end{equation*} 
for every open set $\Omega' \Subset \Omega$. We take the chance to provide a new proof of this bound, given that there is a minor gap in the proof of \cite[Proposition 4.4, eq. (4.4)]{CDM}. We notice that, by \cite[Proposition 4.4, eq. (4.3)]{CDM} and by \eqref{f:trA}, we obtain 
\begin{align*}
(\A, Du)_{\lambda} & = (\A, D u) + \left ( \frac{1}{2} - \lambda \right) (u^+ - u^-) \Div \A \res J_u \\
& = (\A, D u) + \left ( \frac{1}{2} - \lambda \right) (u^+ - u^-) \left ( \Trp{\A}{J_u} - \Trm{\A}{J_u} \right ) \Haus{N-1} \res J_u \\
& = (\A, D u) + \left ( \frac{1}{2} - \lambda \right) \left ( \Trp{\A}{J_u} - \Trm{\A}{J_u} \right ) |D^j u|.
\end{align*}
This implies that the diffuse part of the $\lambda$-pairing satisfies
\begin{equation*}
(\A, Du)_\lambda^d = (\A, Du)^d,
\end{equation*}
while the singular part is given by
\begin{equation*}
(\A, Du)_\lambda^j = (\A, D u)^j + \left ( \frac{1}{2} - \lambda \right) \left ( \Trp{\A}{J_u} - \Trm{\A}{J_u} \right ) |D^j u|.
\end{equation*}
Hence, we can argue as in the proof of \cite[Proposition 4.7]{CDM}: we exploit \cite[Theorem 3.3]{CD3} to get
\begin{align*}
(\A, Du)_{\lambda} & = (\A, D u)^d + \left ( (1 - \lambda) \Trp{\A}{J_u} + \lambda \Trm{\A}{J_u} \right ) |D^j u| \\
& = (\A, D u) \res (\Omega \setminus J_u) + \left ( (1 - \lambda) \Trp{\A}{J_u} + \lambda \Trm{\A}{J_u} \right ) |D u| \res J_u,
\end{align*} 
so that, by applying \eqref{improved_bound} and \eqref{eq:pairing_tot_var_bound} (restricted to $\Omega \setminus J_u$), we obtain
\begin{align*}
|(\A, Du)_\lambda| \res \Omega'  & \le \|\A\|_{L^{\infty}(\Omega'; \R^N)} |D u| \res (\Omega' \setminus J_u) + \|\A\|_{L^{\infty}(\Omega'; \R^N)} |D u| \res (\Omega' \cap J_u) \\
& = \|\A\|_{L^{\infty}(\Omega'; \R^N)} |Du| \res \Omega'
\end{align*}
for every open set $\Omega' \Subset \Omega$.
\end{remark}

We point out that Theorem \ref{t:pairing} gives a complete answer concerning the density of the pairing between $\A$ and characteristic functions of sets of finite perimeter.

\begin{corollary}\label{c:paironch}
If \(\A\in\mathcal{DM}^{\infty}_{\rm loc}(\Omega)\) and $E$ is a set of locally finite perimeter in $\Omega$,
then 
\[
\polar(\A,D\chi_E,x)=\Trace[*]{\A}{\partial^* E}(x) \quad 
\text{ for } \Haus{N-1}\text{-a.e. } x\in \partial^* E.
\]
\end{corollary} 

A complete representation can also be given when $\A$ is a $BV$ vector field: we state below a localization of \cite[Remarks 3.4 and 3.6]{CD3}.

\begin{corollary}\label{r:chiE}
If $\A\in BV_{\rm loc}(\Omega; \R^N)\cap L^{\infty}_{\rm loc}(\Omega; \R^N)$ and $u\in BV_{\rm loc}(\Omega)\cap L^{\infty}_{\rm loc}(\Omega)$, then $\pair{\A, Du}=\A^*\cdot Du$ in the sense of Radon measures in $\Omega$.

In particular, $\Trace[*]{\A}{J_u}(x)=\A^{*}(x) \cdot \nu_u(x)$ for $\Haus{N-1}$-a.e. $x\in J_u$.
\end{corollary}

Theorem~\ref{t:pairing} gives a complete representation of the Cantor
part of the pairing measure only if $|D^c u| (S_{\A}) = 0$.
This requirement
could be an effective restriction to the applicability of Theorem~\ref{t:pairing}.
Indeed, although $\LLN(S_{\A}) = 0$ thanks to Lebesgue differentiation theorem,
the Hausdorff dimension of $S_{\A}$ can be equal to $N$
(hence $|D^c u|(S_{\A})$ can be arbitrarily large),
as it is shown in the following example.

\begin{example}\label{dimensionSA}
For every $N \ge 2$ we construct a vector field \(\A\in\mathcal{DM}^{\infty}(\R^N)\)
such that $\Div \A = 0$ and $\dim_{\Haus{}}(S_{\A}) = N$, where $\dim_{\Haus{}}$ is the Hausdorff dimension.

As a first step we exhibit a set $E \subset \R^{N - 1}$ such that $\dim_{\Haus{}}(\partial^{M} E) = N - 1$, where $\partial^{M} E$ is the measure theoretic boundary of $E$ defined in \eqref{f:MTB}.


We start by considering suitable fat Cantor sets on $\R$. 
Following the construction of Falconer \cite[Example 4.5]{Falconer}, 
for any $\lambda \in (0, 1)$ we can construct a middle third Cantor set on 
$[0, 1]$, removing at each step a proportion $\lambda$ from the intervals. 
In this way, for any $j \ge 0$ we remove from $[0, 1]$ a family of 
middle open intervals $\{I_{j}^{k}\}_{k = 1}^{2^{j}}$ with length 
\[
|I_{j}^{k}| = \lambda \frac{(1 - \lambda)^{j}}{2^{j}}.
\] 
Let us consider the union of the intervals corresponding to
even generations $j$:
\[
E_\lambda := \bigcup_{j=0}^{\infty}\left( \bigcup_{k=1}^{2^{2j}} I_{2j}^k
\right).
\]
Then the fat Cantor set 
\[
C_{\lambda} := [0,1]\setminus \bigcup_{j=0}^{\infty}\left( \bigcup_{k=1}^{2^{j}} I_{j}^k
\right),
\]
coincides with $\partial^M E_\lambda$.
Specifically, reasoning as in \cite[Example 3.5]{CD3},
we can prove that
\[
0 < \liminf_{r\downarrow 0} \frac{|E_\lambda \cap B_r(x)|}{2r}
\leq
\limsup_{r\downarrow 0} \frac{|E_\lambda \cap B_r(x)|}{2r} < 1
\qquad \forall x\in C_\lambda\,.
\]
Moreover, it is known that 
$\dim_{\Haus{}}(C_{\lambda}) = \frac{\log{2}}{\log{\left (\frac{2}{1 - \lambda} \right )}}$ (see for instance \cite[Example 4.5]{Falconer}),
so that, choosing $\lambda = 1 - 2^{1-1/s}$, we have
$\dim_{\Haus{}}(\partial^M E_{\lambda}) = \dim_{\Haus{}}(C_{\lambda}) = s$.

\smallskip
We can now set
\[
F := \bigcup_{m = 1}^{+ \infty} (2m + E_{2^{-m}}) \text{ and } E = F \times \R^{N - 2},
\]
and conclude that 
\begin{equation} 
\label{full_dim_E} 
\dim_{\Haus{}}(\partial^{M} E) = N-1,
\end{equation} 
as claimed.
Indeed, it is clear that $\dim_{\Haus{}}(\partial^{M} F) \le 1$, and
we have
\[
\partial^{M} F = \bigcup_{m = 1}^{+ \infty} (2m + \partial^{M}E_{2^{-m}}),
\] 
which implies
$
\Haus{\alpha}(\partial^{M} F) \ge \Haus{\alpha}(\partial^{M} E_{2^{-m}})$,
for any $m \ge 1$ and $\alpha \in (0, 1)$. 
Since
\[
\dim_{\Haus{}}(\partial^{M} E_{2^{-m}}) =  \frac{\log{2}}{\log{\left (\frac{2}{1 - 2^{-m}} \right )}}, 
\]
there exists $m$ large enough such that 
$\dim_{\Haus{}}(\partial^{M} E_{2^{-m}}) > \alpha, $
for any fixed $\alpha \in (0, 1)$. This shows that $\dim_{\Haus{}}(\partial^{M} F) = 1$, and so we obtain \eqref{full_dim_E}, by \cite[Corollary 7.4]{Falconer}.

Finally,
let us define $\A(x) \equiv \A(x', x_{N}) := (0, \dots, 0, \chi_{E}(x'))$, where
$x = (x', x_N) \in \R^N$. 
It is clear that $\Div \A = 0$, so that $\A \in \DM(\R^N)$ 
and $S_{\A} = \partial^{M} E \times \R$. 
Hence, it follows that $\dim_{\Haus{}}(S_{\A}) = N$, see again \cite[Corollary 7.4]{Falconer}.
\end{example}

\bigskip

Aiming to give a general representation of the Cantor part of the pairing measure
also on $S_{\A}$, 
we are going to use the following coarea formula for 
the pairing, for which we refer to \cite[Theorem 4.2]{CD3} and \cite[Theorem 5.1]{CDM}.

\begin{theorem} \label{t:coarea}
Let $\A\in\DMlocloc[\Omega]$ and let $u\in\BVAlocloc$.
	Then
	\begin{equation*}
	\pscal{\pair{\A, Du}}{\varphi} = \int_{\R} \pscal{\pair{\A, 
	D\chiut}}{\varphi}\, dt,
	\qquad\forall \varphi\in C_c(\Omega)\,,
	\end{equation*}
and, for every Borel set $B\Subset\Omega$,
	\begin{equation*}
	\pair{\A, Du}(B) = \int_{\R} \pair{\A, 
	D\chiut}(B)\, dt.
	\end{equation*}
\end{theorem}

Thanks to this result, we obtain the following straightforward localization of \cite[Proposition~5.2]{CDM} in the case $\lambda \equiv 1/2$.

\begin{proposition}
\label{p:theta}
Let $\A\in\mathcal{DM}^{\infty}_{\rm loc}(\Omega)$ and $u\in BV_{\rm loc}(\Omega) \cap L^{\infty}_{\rm loc}(\Omega)$.
Then for $\mathcal{L}^1$-a.e. $t\in\R$ we have
\begin{equation*}
\polar(\A, Du, x) =
\polar(\A, D\chiut, x)
\quad
\text{for \(|D\chiut|\)-a.e.}\ x\in\Omega\,.
\end{equation*}
\end{proposition}

In the following theorem, the pairing is characterized in terms of normal 
traces of the field $\A$ on the level sets of $u$, without any assumption on $S_{\A}$.

\begin{theorem}\label{t:repr}
Let $\A \in \mathcal{DM}^{\infty}_{\rm loc}(\Omega)$ and $u\in BV_{\rm loc}(\Omega) \cap L^{\infty}_{\rm loc}(\Omega)$.
Then,
\begin{equation}\label{f:reprC}
\theta(\A, Du, x) =
\Trace[*]{\A}{\partial^* \{u>\widetilde u(x)\} }(x),
\qquad
\text{for $|D^du|$-a.e.\ $x\in\Omega$},
\end{equation}
and
\begin{equation}\label{f:reprJ}
\polar(\A,Du,x) = \Trace[*]{\A}{J_u}(x)\,,
\qquad 
\text{for $|D^j u|$-a.e.}\ x\in\Omega.
\end{equation}
\end{theorem}

\begin{proof}
Let us prove \eqref{f:reprC},
since~\eqref{f:reprJ} follows
from Theorem~\ref{t:pairing}(ii). In addition, thanks to the local nature of the statement, without loss of generality we can assume $\A \in \mathcal{DM}^{\infty}(\Omega)$ and $u\in BV(\Omega) \cap L^{\infty}(\Omega)$.

Let $Z\subset\R$ be the set such that
for every $t\in \R\setminus Z$ the following hold:
\begin{itemize}
\item[(a)]
$E_t := \{u > t\}$ is of finite perimeter in $\Omega$;
\item[(b)]
$\hh\Bigl(\{x\in C_u\colon \ut(x)=t \}\setminus
\bigl(C_u\cap\partial^*\{u>t\}\bigr)\Bigr)=0$;
\item[(c)]
$\displaystyle \polar(\A, Du, x) =
\polar(\A, D\chi_{E_t}, x)
= \Trace[*]{\A}{\partial^* E_t }(x)
\quad
\text{for $\hh$-a.e.}\ x\in \partial^* E_t.
$
\end{itemize}
By the coarea formula in $BV$ (Theorem~\ref{coarea}),
formula \eqref{f:incl2},
Proposition~\ref{p:theta}, and 
Corollary~\ref{c:paironch},
we have that $\LLU(Z) = 0$.

\smallskip
Since $\LLU(Z) = 0$,
by \cite[Proposition~3.92(a)(c)]{AFP}, we have that
\[
\nabla u = 0 \quad \text{$\LLN$-a.e.\ in}\ u^{-1}(Z) \, \text{ and } \, 
|D^c u| (\ut^{-1}(Z)) = 0.
\]
As a consequence,
for $|D^d u|$-a.e.\ $x$ we have that
$\ut(x) \in \R\setminus Z$,
i.e.\ $|D^d u|(\ut^{-1}(Z)) = 0$.

For every $t\in\R\setminus Z$, let $N_t\subset \partial^* E_t$
be a set such that
the following hold:
\begin{itemize}
\item[(d)]
$\ut(x) = t$ for every $x\in C_u \cap (\partial^* E_t \setminus N_t)$;
\item[(e)]
equality in (c) holds for every $x\in \partial^* E_t \setminus N_t$.
\end{itemize}
By (b) and (c), the set $N_t$ can be chosen of zero
$\hh$ measure.

We claim that
\begin{equation}
\label{f:B}
|D^d u|(\Omega \setminus B) = 0,
\qquad
\text{where}\
B := \bigcup_{t\in\R\setminus Z} ( \partial^* E_t \setminus N_t)\,.
\end{equation}
Specifically, 
since the sets $\partial^* E_t \cap C_u$, $t\in\R\setminus Z$, are pairwise disjoint
(see \cite[p.~356]{GMS1}),
we have that
$
\partial^* E_t \cap (C_u \setminus B) = C_u \cap N_t$,
hence, by the coarea formula for $BV$ functions,
\begin{align*}
|D^d u|(\Omega\setminus B) = |Du|(C_u\setminus B)
& = \int_{\R\setminus Z} \hh(\partial^* E_t \cap (C_u \setminus B))\, dt \\
& \le \int_{\R\setminus Z} \hh(N_t)\, dt = 0.
\end{align*}

Finally, for every $x\in B \cap C_u$
(hence, by \eqref{f:B}, for $|D^d u|$-a.e.\ $x\in\R^N$),
we have that $x \in \partial^* E_{\ut(x)}$ 
and \eqref{f:reprC} holds.
\end{proof}

\begin{remark}
As a consequence of Theorem~\ref{t:repr}, the following new representation formula holds for the Cantor part of the pairing measure:
\begin{equation*}
(\A, Du)^c = \Trace[*]{\A}{\partial^* \{u>\widetilde u(\cdot)\} }(\cdot) |D^cu|.
\end{equation*}
In Section \ref{Sec:repr} we will provide more explicit representations of $\Trace[*]{\A}{\partial^* \{u>\widetilde u(\cdot)\} }(\cdot)$ (Remark \ref{rem:Cantor_1} and Corollary \ref{c:repr}).
\end{remark}

\begin{remark}\label{rem:lambda_repr_case}
If we define $\theta_{\lambda}(\A, Du, x)$ to be the density of the $\lambda$-pairing $(\A, Du)_\lambda$ (defined by \eqref{def:lambda_pair}) with respect to $|Du|$, then, in light of \cite[Proposition 4.7]{CDM}, identity \eqref{f:reprC} in Theorem~\ref{t:repr} remains true for $\theta_{\lambda}(\A, Du, x)$, while \eqref{f:reprJ} becomes $$\theta_{\lambda}(\A, Du, x) = (1-\lambda(x))\Trp{\A}{J_u}(x) +\lambda(x) \Trm{\A}{J_u}(x) \qquad \text{for $|D^j u|$-a.e.}\ x\in\Omega.$$
\end{remark}


\section{Other representation formulas}\label{Sec:repr}

For vector fields with $L^1$ divergence, 
an explicit representation of the density $\theta(\A, Du, x)$
in terms of cylindrical averages has been proposed
in the unpublished paper \cite{Anz2}.

More precisely, in \cite[Theorem~3.6]{Anz2}  it is established that, if
$\Div \A \in L^1(\Omega)$
and $u\in BV(\Omega)\cap L^\infty(\Omega)$, then
\begin{equation}\label{f:denscyl}
\theta(\A, Du, x) = 
\Cyl{\A}{\nu_{u}}(x)
\qquad
\text{for $|Du|$-a.e.}\ x\in\Omega,
\end{equation}
where, for some set $G \subset \Omega$ and some function $\zeta : G \to \mathbb{S}^{N-1}$,
\begin{equation*}
\Cyl{\A}{\zeta}(x)
:= \lim_{\rho\downarrow 0}
\lim_{r\downarrow 0}
\frac{1}{\LLN(C_{r,\rho}(x, \zeta(x)))}
\int_{C_{r,\rho}(x, \zeta(x))} \A(y) \cdot \zeta(x) \, dy \quad \text{ for } x \in G
\end{equation*}
whenever the limits exist, with
\[
C_{r,\rho}(x, \zeta(x)) :=
\left\{
y\in\R^N:\ 
|(y-x)\cdot\zeta(x)| < r,\
|(y-x) - [(y-x)\cdot\zeta(x)]\zeta(x)| < \rho
\right\}
\]
(the existence of the limit in the definition of $\Cyl{\A}{\nu_u}$
for $|Du|$-a.e.\ $x\in\Omega$ is part of the statement).

We will extend this result by adapting the arguments of the proof contained 
in \cite{Anz2} to the general case by means of properties of the pairing
recently obtained in \cite{CD3}. 
We will obtain (see Theorem \ref{representAnz} below) that the same formula holds for a general divergence-measure field $\A$ by assuming the weaker condition
$\Haus{N-1}(\Theta_{\A}\cap J_u)=0$,
where $\Theta_{\A}$ is the jump set of $\Div \A$ defined
in \eqref{f:jump}.

If $\Div \A$ has a non-vanishing jump part concentrated on the jump set $J_u$ of $u$, 
the relation~\eqref{f:denscyl} is no longer valid, as it is shown in the following example.

\begin{example}
\label{e:no}
Let $\A\in\DM(\R^2)$ be the $BV$ vector field $\A(x) = \textbf{a}\,\chi_{B_1}(x)$,
where $\textbf{a} \in \R^2$ is a fixed vector, and let
$u := \chi_{B_1} \in BV(\R^2)$.
By Corollary~\ref{r:chiE} we have that
\[
\pair{\A, Du} = \frac{1}{2}\, \textbf{a}\cdot \nu(x) \, 
\mathcal{H}^1 \res \partial B_1,
\qquad \nu(x) = - \frac{x}{|x|}\,,
\]
hence $\polar(\A, Du, x) = \frac{1}{2}\,\textbf{a}\cdot\nu(x)$, $x\in\partial B_1$.
On the other hand, for every $x\in\partial B_1$ 
and every $0 < r \ll \rho$, an explicit computation gives
\[
\begin{split}
\int_{C_{r,\rho}(x, \nu)} \A(y) \cdot \nu(x)\, dy
& = \left[\arccos(1-r) - (1-r) \sqrt{2r - r^2}\right]\,
\textbf{a}\cdot \nu(x)
\\ & =
\frac{4\sqrt{2}}{3}\, r^{3/2}\, \textbf{a}\cdot \nu(x)
+ o (r^{3/2})\,,
\end{split}
\]
hence
\[
\Cyl{\A}{\nu}(x)
= \lim_{\rho \to 0} \lim_{r\to 0}
\frac{\sqrt{2} r^{1/2}}{3 \rho} \, \textbf{a}\cdot \nu(x) = 0\,.
\]
As a consequence, if $\textbf{a}\neq \textbf{0}$ we have that
$\polar(\A, Du, x) \neq \Cyl{\A}{\nu}(x)$ for all but two points
$x\in\partial B_1$. 
\end{example} 

However, we notice that it is still possible to achieve a representation for the pairing density $\polar(\A, Du, x)$ even in the case in which $\Haus{N-1}(\Theta_{\A} \cap J_u) > 0$: to this purpose we will exploit the averages on half-balls introduced in \cite{Silh}, see Theorem \ref{c:repr} below.

\medskip

For the reader's benefit, we state below the general Gauss-Green formulas for essentially bounded divergence-measure, sets of finite perimeters and functions of bounded variation, which is the localized version of \cite[Theorem 6.3]{CDM} in the case $\lambda \equiv 1/2$, and whose proof we leave to the interested reader.

\begin{theorem}[Gauss-Green formula]
\label{t:gGG}
Let $\A \in \mathcal{DM}^{\infty}_{\rm loc}(\Omega)$ and $u\in \BVAlocloc$.
Let \(E\) be a set with locally finite perimeter in $\Omega$ such that $\supp(\chi_E u) \Subset \Omega$. Assume that the traces $u^i, u^e$ of $u$ on $\partial^* E$ belong to $L^1_{\rm loc}(\partial^*E, \Haus{N-1}\res \partial^* E)$.
Then 
the following Gauss--Green formulas hold:
\begin{gather}
\int_{E^1} \prec{u} \, d\Div\A + \int_{E^1} d\pair{\A, Du} = -
\int_{\partial ^*E} \uint \, \Trp{\A}{\partial^* E} \, d\mathcal 
H^{N-1}\,,\label{gGreenIB}
\\
\int_{E^1\cup \partial ^*E} \prec{u} \, d\Div\A + \int_{E^1\cup \partial ^*E} 
d\pair{\A, Du} = -
\int_{\partial ^*E} \uext \, \Trm{\A}{\partial^* E} \, d\mathcal 
H^{N-1}\,,\label{gGreenIB2}
\end{gather}
where $E^1$ is the measure theoretic interior of $E$.
\end{theorem}

In the particular case of $u \equiv 1$, Theorem \ref{t:gGG} reduces to
\begin{gather}
\Div\A(E^1) = -
\int_{\partial ^*E} \Trp{\A}{\partial^* E} \, d\mathcal 
H^{N-1}\,,\label{gGreenIB_u_1}
\\
\Div\A(E^1\cup \partial ^*E) = -
\int_{\partial ^*E} \Trm{\A}{\partial^* E} \, d\mathcal 
H^{N-1}\,,\label{gGreenIB2_u_1}
\end{gather}
and these formulas are fundamental tools needed in order to generalize \eqref{f:denscyl} (see the proof of Lemma \ref{l:limsup} below). The proof of \eqref{gGreenIB_u_1} and \eqref{gGreenIB2_u_1} can be found for instance in \cite[Theorems 4.1 and 4.2]{ComiPayne}.

The following technical lemma, generalizing \cite[Theorem 3.3]{Anz2} in the case of a singular measure $\Div \A$,
gives an estimate of the gap between the local behavior of the mean 
values of the normal traces of $\A$ on a smooth surface $\Sigma$ and their 
analogous computed on tangent hyperplanes $T_x \Sigma$ to $\Sigma$. By means of the Gauss--Green formulas, we show that the gap is possibly due to the concentration of the measure $|\diver \A|$ on $\Sigma$.

\begin{lemma}\label{l:limsup}
Let $\A \in \mathcal{DM}^{\infty}_{\rm loc}(\Omega)$,
and let $\Sigma\subset\R^N$ be an oriented 
$C^1$ hypersurface with classical normal vector field $\nu_\Sigma$.
Then,
for every $x\in\Sigma \cap \Omega$,
\begin{equation}
\label{f:limsup}
\begin{split}
\limsup_{\rho\to 0^+} & \Bigg|
\frac1{\omega_{N-1}\rho^{N-1}}\int_{T_x\Sigma\cap B_\rho(x)}
\Trie{\A}{T_x\Sigma}(y)\,d\mathcal H^{N-1}(y)
\\ & -
\frac1{\hh(\Sigma\cap B_\rho(x))}\int_{\Sigma \cap B_\rho(x)}
\Trie{\A}{\Sigma}(y)\,d\mathcal H^{N-1}(y)
\Bigg| 
\\ & \leq 
\limsup_{\rho\to 0^+}
\frac{|\diver\A|(B_\rho(x))}{\omega_{N-1}\, \rho^{N-1}},
\end{split}
\end{equation}
where $\Trie{\A}{\cdot}$ denotes either $\Tr^{i}(\A, \cdot)$ or $\Tr^e(\A, \cdot)$, and $T_{x}\Sigma$ is the tangent hyperplane to $\Sigma$ in $x$.
\end{lemma}

\begin{proof}
Let us prove \eqref{f:limsup} for $\Tr^i$, being the
computation for $\Tr^e$ entirely similar.

Up to a change of coordinates, we may assume that $x=0 \in \Omega$, $\nu_\Sigma(0) = - e_N$,
so that $\Sigma$ is locally the graph of a
$C^1$ function $\varphi\colon \R^{N-1}\to \R$ with
$\varphi(0) = 0$, $\nabla\varphi(0) = 0$.
For every $\rho > 0$ such that $B_{2 \rho} \Subset \Omega$, we denote
\begin{gather*}
B_\rho^- := B_\rho\cap \{x_N<0\},
\quad 
S_\rho^- := \partial B_\rho \cap \{x_N < 0\},
\quad
T_{\rho} := B_\rho\cap\{x_N = 0\},
\\ 
\Sigma^- := \{x = (x', x_N) \in \R^N:\ x_N < \varphi(x')\}\,,
\\
\Sigma_\rho := \Sigma\cap B_\rho,
\quad
E_\rho := \Sigma^- \cap B_\rho,
\quad
\Sigma_\rho^- := \Sigma^-\cap \partial B_\rho.
\end{gather*}
Now we apply the Gauss-Green formula \eqref{gGreenIB_u_1} to $\A$ and to the open piecewise Lipschitz set 
$E_\rho$, and we obtain
\begin{equation}\label{gree1}
\begin{split}
\Div\A(E_\rho) & = -
\int_{\partial E_\rho} \Trp{\A}{\partial E_\rho} \ d\mathcal H^{N-1}
\\ & =
-
\int_{\Sigma_\rho} \Trp{\A}{\Sigma} \, d\mathcal H^{N-1}
-
\int_{\Sigma_\rho^-} \Trp{\A}{\partial B_\rho} \ d\mathcal H^{N-1},
\end{split}
\end{equation}
since $\Haus{N-1}(E_{\rho}^1 \setminus E_{\rho}) = 0$, $\Haus{N-1}(\partial E_\rho \setminus \partial^* E_{\rho}) = 0$, $\Haus{N-1}(\partial E_\rho \setminus (\Sigma_\rho \cup \Sigma_\rho^-)) = 0$ and thanks to the locality of the normal traces \eqref{locality_traces}.
Similarly, we apply the Gauss-Green formula \eqref{gGreenIB_u_1} to $\A$ and to $B_\rho^-$, so that we get
\begin{equation}\label{gree2}
\Div\A(B_\rho^-) = 
-
\int_{T_\rho} \Trp{\A}{T_\rho} \, d\mathcal H^{N-1}
-
\int_{S_\rho^-} \Trp{\A}{\partial B_\rho} \, d\mathcal H^{N-1},
\end{equation}
since $\Haus{N-1}((B_\rho^-)^1 \setminus B_\rho^-) = 0$, $\Haus{N-1}(\partial B_\rho^- \setminus \partial^* B_\rho^-) = 0$, $\Haus{N-1}(\partial B_\rho^- \setminus (T_\rho \cup S_\rho^-)) = 0$, and again thanks to \eqref{locality_traces}.
From \eqref{gree1} and \eqref{gree2} we obtain that
\begin{equation}\label{f:estim1}
\begin{split}
\Bigg|\int_{\Sigma_\rho} \Trp{\A}{\Sigma} & \, d\mathcal H^{N-1} 
-\int_{T_\rho} \Trp{\A}{T_\rho} \, d\mathcal H^{N-1}
\Bigg|
\\ & \leq 
\left|\Div\A(B_\rho^-)-\Div\A(E_\rho)\right|
\\ & \quad +
\left|\int_{S_\rho^-} \Trp{\A}{\partial B_\rho} \, d\mathcal H^{N-1}
-
\int_{\Sigma_\rho^-} \Trp{\A}{\partial B_\rho} \ d\mathcal H^{N-1}
\right|
\\ & \leq
|\Div\A|(B_\rho^- \vartriangle E_\rho)
+ \|\A\|_{L^{\infty}(B_{2 \rho}; \R^N)} \, \hh(S_\rho^- \vartriangle \Sigma_\rho^-)\,,
\end{split}
\end{equation}
where in the last inequality we have used the facts that
$|\mu(B) - \mu(C)| \leq |\mu|(B\vartriangle C)$ for every signed measure $\mu$
and $\|\Trp{\A}{\partial B_\rho}\|_{L^{\infty}(\partial B_\rho, \Haus{N-1} \res \partial B_\rho)} \leq \|\A\|_{L^{\infty}(B_{2 \rho}; \R^N)}$, thanks to \eqref{improved_bound}.

Since $\Sigma$ is a $C^1$ hypersurface, then $$\lim_{\rho \to 0^+} \frac{\hh(\Sigma_\rho)}{\omega_{N-1}\, \rho^{N-1}} = 1,$$ and, since $\nabla\varphi(0) = 0$, we have that
\[
\hh(S_\rho^- \vartriangle \Sigma_\rho^-) =
o(\rho^{N-1}),
\]
so that, substituting in \eqref{f:estim1},
\[
\begin{split}
& \Bigg|\frac{1}{\hh(\Sigma_\rho)}\int_{\Sigma_\rho} \Trp{\A}{\Sigma}  \, d\mathcal H^{N-1} 
-\frac{1}{\omega_{N-1}\, \rho^{N-1}}\int_{T_\rho} \Trp{\A}{T_\rho} \, d\mathcal H^{N-1}
\Bigg|
\\ & \leq 
\frac{1}{\omega_{N-1}\, \rho^{N-1}}
\left|
\int_{\Sigma_\rho} \Trp{\A}{\Sigma}  \, d\mathcal H^{N-1} 
-
\int_{T_\rho} \Trp{\A}{T_\rho} \, d\mathcal H^{N-1}
\right|
\\ & \quad +
\left|\left(\frac{1}{\hh(\Sigma_\rho)} - \frac{1}{\omega_{N-1}\, \rho^{N-1}}\right)
\int_{\Sigma_\rho} \Trp{\A}{\Sigma}  \, d\mathcal H^{N-1} 
\right|
\\ & \leq
\frac{|\Div\A|(B_\rho)}{\omega_{N-1}\, \rho^{N-1}}
+ \|\A\|_{L^{\infty}(B_{2 \rho}; \R^N)}\, \frac{o(\rho^{N-1})}{\omega_{N-1}\, \rho^{N-1}}
+ \|\A\|_{L^{\infty}(B_{2 \rho}; \R^N)}\,
\left|
1 - \frac{\hh(\Sigma_\rho)}{\omega_{N-1}\, \rho^{N-1}}
\right|\,,
\end{split}
\]
and hence \eqref{f:limsup} follows.
\end{proof}

We briefly recall the behaviour of the cylindrical averages on hyperplanes.

\begin{lemma}
\label{l:traces}
Let $\A \in \mathcal{DM}^{\infty}_{\rm loc}(\Omega)$ and
let $T\subset\R^N$ be a hyperplane oriented with normal vector $\nu$.
Then, for every $x\in T \cap \Omega$ and every $\rho > 0$ such that $B_\rho(x) \Subset \Omega$,
\[
\int_{T \cap B_\rho(x)} \Trp{\A}{T}\, d\hh =
\lim_{r\downarrow 0} \frac{1}{r}
\int_{C_{r,\rho}(x, \nu)} \A(y)\cdot \nu\, dy.
\]
\end{lemma}

\begin{proof}
To simplify the notation,
assume that $T = \{x_N = 0\}$, $\nu = e_N$, $x = 0 \in \Omega$,
and let $T_t := \{x_N = t\}$ for $t \in J := \{t \in \R : T_t \cap \Omega \neq \emptyset \}$.
By \cite[Proposition~3.6]{AmbCriMan} we  have that
\[
\A \cdot e_N = \Trp{\A}{T_t} = \Trm{\A}{T_t},
\qquad
\text{$\hh$-a.e.\ on $T_t \cap \Omega$,
for $\Leb{1}$-a.e.\ $t\in J$},
\]
so that, for every $r>0$ such that $C_{r, \rho}(0, e_N) \Subset \Omega$,
\begin{align*}
\frac{1}{r} \int_{C_{r, \rho}(0, e_N)} \A(y)\cdot e_N\, dy
& = \frac{1}{r} \int_0^r \int_{T_t \cap B_\rho} \A(y)\cdot e_N \, d\hh(y)\, dt \\
& = \frac{1}{r} \int_0^r \int_{T_t \cap B_\rho} 
\Trm{\A}{T_t}(y)\, d\hh(y)\, dt \nonumber \\
& = \frac{1}{r} \int_0^r \int_{\mathcal{D}_\rho} 
\Trm{\A}{T_t}(y', t)\, d \Leb{N-1}(y') \, dt \\
& = \int_0^1 \int_{\mathcal{D}_\rho} 
\Trm{\A}{T_{rt}}(y', r t)\, d \Leb{N-1}(y') \, dt,
\end{align*}
where $\mathcal{D}_\rho$ is the $(N-1)$-dimensional disk of radius $\rho$ centered in the origin, which satisfies $T \cap B_\rho = \{ (y', 0) : y' \in \mathcal{D}_\rho \}$.
Moreover, by \cite[Theorem~3.7]{AmbCriMan}, we have that
\[
\lim_{s \to 0^+} \Trm{\A}{T_s} =
\Trp{\A}{T}\,,
\qquad
w^\ast - L^\infty(\mathcal{D}_\rho, \Leb{N-1} \res \mathcal{D}_\rho).
\]
This implies that 
\begin{align*}
\lim_{r \to 0^+} \int_0^1 \int_{\mathcal{D}_\rho} 
\Trm{\A}{T_{rt}}(y', r t)\, d \Leb{N-1}(y') \, dt & = \int_{\mathcal{D}_\rho} 
\Trp{\A}{T}(y', 0)\, d \Leb{N-1}(y') \\
& = \int_{T \cap B_\rho(x)} \Trp{\A}{T}(y) \, d\hh(y),
\end{align*}
which ends the proof.
\end{proof}

Finally, we are able to specify where the cylindrical averages of the field $\A$ on 
oriented rectifiable sets coincide with its weak normal traces.

\begin{theorem}\label{t:trcyl}
Let $\A \in \mathcal{DM}^{\infty}_{\rm loc}(\Omega)$.
Let $\Sigma\subset\R^N$ be an oriented countably
$\hh$-rectifiable set.
Then,
for $\hh$-a.e.\ $x\in \Omega \cap \Sigma\setminus \Theta_{\A}$,
\[
\Trp{\A}{\Sigma}(x) =
\Trm{\A}{\Sigma}(x) =
\Cyl{\A}{\nu_\Sigma}(x).
\]
\end{theorem}

\begin{proof}
As a first step, let us prove the theorem in the case $\Sigma$ is an oriented $C^1$ hypersurface, with (classical) normal vector field $\nu_\Sigma$.

Let $x\in \Omega \cap \Sigma\setminus \Theta_{\A}$ be a Lebesgue point for both
$\Trp{\A}{\Sigma}$ and $\Trm{\A}{\Sigma}$
with respect to $\hh\res\Sigma$.
From Lemma~\ref{l:limsup}, and recalling the definition \eqref{f:jump}
of $\jump{\A}$, we deduce that there exist the limits
\begin{equation}\label{f:trlim}
\lim_{\rho\downarrow 0}
\frac1{\omega_{N-1}\rho^{N-1}}\int_{T_x\Sigma\cap B_\rho(x)}
\Trie{\A}{T_x\Sigma}(y)\,d\mathcal H^{N-1}(y)
= \Trie{\A}{\Sigma}(x)\,.
\end{equation}
On the other hand, it holds that
\begin{equation*}
\Trp{\A}{\Sigma}=\Trm{\A}{\Sigma}
\qquad
\text{$\mathcal H^{N-1}$-a.e. in $\Omega \cap \Sigma\setminus \jump{\A}$}.
\end{equation*}
Specifically,
by \eqref{f:trA}, we have that
\[
(\Trp{\A}{\Sigma}-\Trm{\A}{\Sigma}) \hh\res (\Omega \cap \Sigma\setminus\jump{\A})
= \diver\A \res (\Omega \cap \Sigma\setminus\jump{\A})
= 0.
\]
From \eqref{f:trlim} and Lemma~\ref{l:traces} we deduce that,
for $\hh$-a.e.\ $x\in \Omega \cap \Sigma\setminus\jump{\A}$,
\[
\Trp{\A}{\Sigma}(x) =\Trm{\A}{\Sigma}(x) =
\lim_{\rho\downarrow 0} \lim_{r\downarrow 0} 
\frac1{\omega_{N-1}\rho^{N-1} r}
\int_{C_{r,\rho}(x, \nu_\Sigma(x))} \A(y)\cdot \nu_\Sigma(x) \, dy
= \Cyl{\A}{\nu_\Sigma}(x),
\]
hence the claim is proved.

The general case with $\Sigma\subset\R^N$ oriented countably
$\hh$-rectifiable set
follows directly from the previous step and the definition of
weak normal trace on $\Sigma$ (Section \ref{distrtraces}).
\end{proof}

\begin{corollary}\label{3.5} 
Let $\A\in\mathcal{DM}^{\infty}_{\rm loc}(\Omega)$ and let $E$ be a set of locally finite perimeter in $\Omega$. Then for $|D\chi_{E}|$-a.e. $x\in \Omega\setminus\Theta_{\A}$ the limit
\[
\Cyl{\A}{\nu_E}(x)
:=\lim_{\rho\to 0^+}\lim_{r\to 0^+}
\frac{1}{|C_{r,\rho}(x,\nu_{E}(x))|}
\int_{C_{r,\rho}(x,\nu_{E}(x))}
\A(y)\cdot\nu_{E}(x)\,dy
\]
exists, 
and 
\begin{equation*}
\theta(\A,D\chi_{E},x) = \Cyl{\A}{\nu_{E}}(x)
\qquad
\text{for $|D\chi_{E}|$-a.e. $x\in \Omega\setminus\Theta_{\A}$}.
\end{equation*}
\end{corollary}

\begin{remark}
This result has been proved in
\cite[Theorem 3.5]{Anz2} under the stronger assumption $\Div\A\in L^1(\Omega)$.
\end{remark}

\begin{proof}
It is a consequence of Theorem~\ref{t:trcyl}, with
$\Sigma = \partial^* E$,
and of Theorem~\ref{t:repr}. 
\end{proof}

\begin{corollary}\label{coroll} 
Let $\A\in\mathcal{DM}^{\infty}_{\rm loc}(\Omega)$
and let $u\in BV_{\rm loc}(\Omega)\cap L^\infty_{\rm{loc}}(\Omega)$. Let us denote $E_{t} :=\{u>t\}$. For $\mathcal L^1$-a.e. $t\in \R$ and for $|D\chi_{E_{t}}|$-a.e. $x\in \Omega\setminus\Theta_{\A}$
the limit
\[
\Cyl{\A}{\nu_{E_{t}}}(x)
:=\lim_{\rho\to 0^+}\lim_{r\to 0^+}
\frac1{|C_{r,\rho}(x,\nu_{E_{t}}(x))|}
\int_
{C_{r,\rho}(x,\nu_{E_{t}}(x))}
\A(y)\cdot\nu_{E_{t}}(x)\,dy
\]
exists, 
and it holds that
\begin{equation*}
\theta(\A,D\chi_{E_{t}},x) =
\Cyl{\A}{\nu_{E_{t}}}(x) \qquad
\text{for $|D\chi_{E_t}|$-a.e. $x\in \Omega\setminus\Theta_{\A}$}.
\end{equation*}
\end{corollary}

Using Corollary~\ref{3.5}, we get a refinement of Theorem~\ref{t:gGG} under a compatibility condition.
  
\begin{theorem}\label{t:GG}
Let $\A \in \mathcal{DM}^{\infty}_{\rm loc}(\Omega)$ and $u\in \BVAlocloc$.
Let \(E\) be a set with locally finite perimeter in $\Omega$ such that $\supp(\chi_E u) \Subset \Omega$. Assume that the traces $u^i, u^e$ of $u$ on $\partial^* E$ belong to $L^1_{\rm loc}(\partial^*E, \Haus{N-1}\res \partial^* E)$. Assume also that
\begin{equation*}
\Haus{N-1}(\Theta_{\A}\cap \{ x \in \partial^*E : u^{i, e}(x) \neq 0 \})=0.
\end{equation*}
Then the
following Gauss--Green formulas hold:
\begin{gather*}
\int_{E^1} u^* \, d\Div\A + \int_{E^1} d (\A, Du) = -
\int_{\partial ^*E} u^i \ \Cyl{\A}{\nu_{E}} \,
d\mathcal H^{N-1},
\\
\int_{E^1\cup \partial ^*E} u^* \, d\Div\A + 
\int_{E^1\cup \partial ^*E} d (\A, Du) = -
\int_{\partial^*E} u^e\ \Cyl{\A}{\nu_{E}}
\,d\mathcal H^{N-1}.
\end{gather*}
\end{theorem}

Finally, we obtain the following generalization of \cite[Theorem 3.6]{Anz2}.

\begin{theorem} 
\label{representAnz}
Let $\A\in\mathcal{DM}^{\infty}_{\rm loc}(\Omega)$ 
and let $u\in BV_{\rm loc}(\Omega)\cap L^\infty_{\rm{loc}}(\Omega)$.
Then 
\begin{equation}\label{fuori}
(\A, Du) \res (\Omega \setminus\jump{\A})
=
\Cyl{\A}{\nu_{u}}\,|Du|  \res (\Omega \setminus\jump{\A})\,.
\end{equation}
If in addition we assume that
\begin{equation}\label{condit2}
\Haus{N-1}(\Theta_{\A}\cap J_u)=0,
\end{equation}
then 
\begin{equation}\label{f:thetajump}
\theta(\A,Du,x) = \Cyl{\A}{\nu_u}(x),
\qquad
\text{ for $\Haus{N-1}$-a.e. $x \in J_u$}.
\end{equation}
Finally, if $\Omega$ is an open bounded set such that $\Haus{N-1}(\partial \Omega) < \infty$ and $\Haus{N-1}(\partial \Omega \setminus \partial^* \Omega) = 0$, $\A\in\mathcal{DM}^{\infty}(\Omega)$, $u\in BV(\Omega)\cap L^\infty(\Omega)$ and \eqref{condit2} holds, then
\begin{equation}
\label{gaugre}
\int_{\Omega}\,u^* \, d \Div\A+\int_{\Omega}
\Cyl{\A}{\nu_u}\,d |Du| =
-\int_{\partial\Omega}u^i\,\Trp{\widehat{\A}}{\partial\Omega}\,d\Haus{N-1},
\end{equation}
where $\widehat{\A}$ is the zero extension of $\A$ to $\R^N \setminus \Omega$.
\end{theorem}

\begin{proof}
By Corollary~\ref{coroll}, for $\mathcal L^1$-a.e.\ $t\in \R$ and 
for $|D\chi_{E_{t}}|$-a.e.\ $x\in \Omega\setminus\Theta_{\A}$ 
the limit in the definition of
$\Cyl{\A}{\nu_{E_{t}}}(x)$
exists,  and it holds that
\begin{equation}\label{ghghjjj}
(\A,D\chi_{E_{t}}) \res {(\Omega\setminus\Theta_{\A})}
=
\Cyl{\A}{\nu_{E_{t}}}\,|D\chi_{E_{t}}|
 \res {(\Omega\setminus\Theta_{\A})}\,.
\end{equation}
Since
\begin{equation*}
\frac{D\chi_{E_{t}}}{|D\chi_{E_{t}}|}=\frac{Du}{|Du|} = \nu_u \,,
\quad
\text{$|D\chi_{E_{t}}|$-a.e.\ in}\ \Omega,
\quad
\text{for $\Leb{1}$-a.e.\ $t\in\R$},
\end{equation*}
(see \cite[\S 4.1.4, Theorem~2(i)]{GMS1}), 
for $\mathcal L^1$-a.e. $t\in \R$ and for $|D\chi_{E_{t}}|$-a.e. $x\in \Omega\setminus\Theta_{\A}$ (i.e. for $|Du|$-a.e. $x\in \Omega\setminus\Theta_{\A}$) there exists the limit $\Cyl{\A}{\nu_u}(x)$ and
\begin{equation}\label{cyl}
\Cyl{\A}{\nu_u} =
\Cyl{\A}{\nu_{E_{t}}}
\qquad
\text{$|D\chi_{E_{t}}|$-a.e.\ in}\ \Omega\setminus\Theta_{\A}.
\end{equation}
For every Borel set $B\subset\Omega$,
by the coarea formula (see Theorem \ref{t:coarea}), by \eqref{ghghjjj}, \eqref{cyl}  and the coarea formula in $BV$ (see Theorem~\ref{coarea}), we have that
\[
\begin{split}
(\A, Du) & \res (\Omega\setminus\jump{\A})(B)
=
(\A, Du)(B\setminus\jump{\A})
\\ & =
\int_{\R} (\A, D\chi_{E_{t}})(B\setminus\jump{\A}) \, dt
=
\int_{\R}
\int_{B\setminus\jump{\A}}
\Cyl{\A}{\nu_u} \, d| D\chi_{E_{t}}| \, dt
\\
& =
\int_{B\setminus\jump{A}}
\Cyl{\A}{\nu_u}\,d |Du|\,,
\end{split}
\]
so that \eqref{fuori} holds.

In order to prove \eqref{f:thetajump}, we notice that the assumption \eqref{condit2} implies that $|D^j u|(\Theta_{\A})=0$, so that $(\A,D^ju)(\Theta_{\A})=0$. 
Therefore, \eqref{fuori} implies that
\begin{equation*}
(\A,D^ju)=\Cyl{\A}{\nu_u}\, |D^j u|.
\end{equation*}
Finally, we deduce \eqref{gaugre} by extending $\A$ and $u$ to zero on $\R^N \setminus \Omega$, and then exploiting \cite[Theorems 5.1 and 6.2]{CD4} (see also \cite[Corollary 5.5]{ComiPayne}), the Gauss–Green formula \eqref{gGreenIB} and \eqref{f:thetajump}.
\end{proof}

\begin{remark} \label{rem:Cantor_1}
Recalling that $\jump{\A}$ has \(\sigma\)-finite \(\hh\)-measure,
we have that $|D^du|(\jump{\A})=0$,
and hence
$(\A,Du)^d(\jump{\A})=0$.
Thus, from \eqref{fuori} we deduce that
\begin{equation}\label{fuoriteta2}
(\A,Du)^d = \Cyl{\A}{\nu_u}\, |D^d u|.
\end{equation}
We emphasize that \eqref{fuoriteta2} gives a pointwise representation for the density of the Cantor part $(\A,Du)^c$ of the pairing measure, i.e.
\begin{equation*}
(\A,Du)^c= \Cyl{\A}{\nu_u}\, |D^c u| = \Cyl{\A}{\frac{D^cu}{|D^cu|}}\, |D^c u|,
\end{equation*}
without any assumption on $\jump{\A}$.
\end{remark}

\medskip

The main drawback of the previous representation formula, as showed also in Example \ref{e:no}, is that it fails on the intersection of the jump sets of the divergence-measure and the $BV$ function. In order to circumvent this issue, we conclude this section with one more representation of the pairing, obtained combining Theorem~\ref{t:repr} and the following result  
\cite[Theorem~4.4]{Silh} (see also \cite[Remark 6.4]{ComiPayne} and \cite[Theorem~3]{Silh2019}), representing the normal traces of the 
field $\A$ as limits of averages in half balls.

\begin{theorem}\label{t:sil}
Let $\A \in \mathcal{DM}^{\infty}_{\rm loc}(\Omega)$
and 
let $\Sigma$ be an oriented countably $\hh$-rectifiable set with normal $\nu_\Sigma$. 
Then, for $\hh$-a.e.\ $x\in\Sigma \cap \Omega$
it holds that
\begin{gather}
\Trp{\A}{\Sigma}(x) =
\lim_{r\to 0}
\frac{N}{\omega_{N-1} r^N}
\int_{B_r^i(x,\nu_\Sigma(x))} \A(y) \cdot \frac{y-x}{|y-x|}\, dy,
\label{f:sil1}\\
\Trm{\A}{\Sigma}(x) = 
- \lim_{r\to 0}
\frac{N}{\omega_{N-1} r^N}
\int_{B_r^e(x,\nu_\Sigma(x))} \A(y) \cdot \frac{y-x}{|y-x|}\, dy,
\label{f:sil2}
\end{gather}
where $$B_r^i(x,\nu_\Sigma(x)) := \{y\in B_r(x):\ (y-x)\cdot \nu_\Sigma(x) > 0\}$$ and $$B_r^e(x,\nu_\Sigma(x)) := \{y\in B_r(x):\ (y-x)\cdot \nu_\Sigma(x) < 0\}.$$
\end{theorem}


\begin{corollary}\label{c:repr}
Let $\A \in \mathcal{DM}^{\infty}_{\rm loc}(\Omega)$ and $u\in BV_{\rm loc}(\Omega) \cap L^{\infty}_{\rm loc}(\Omega)$.
Then, for $|Du|$-a.e.\ $x\in\Omega$,
\begin{equation*}
\theta(\A, Du, x) =
\lim_{r\to 0}
\frac{N}{2 \omega_{N-1} r^N}
\left(
\int_{B_r^i(x,\nu_u(x))} \A(y) \cdot \frac{y-x}{|y-x|}\, dy
-
\int_{B_r^e(x,\nu_u(x))} \A(y) \cdot \frac{y-x}{|y-x|}\, dy
\right)\,.
\end{equation*}
In particular,
\begin{equation}\label{f:thetareg}
\theta(\A, Du, x) = \widetilde{\A}(x) \cdot \nu_u(x),
\qquad
\text{for $|D^d u|$-a.e.}\ x\in \Omega\setminus S_{\A}.
\end{equation}
\end{corollary}

\begin{proof}
The first statement
follows from Theorems~\ref{t:repr} and~\ref{t:sil}.
More precisely, if $x\in J_u$ we use \eqref{f:reprJ}
and Theorem~\ref{t:sil} with $\Sigma = J_u$, 
whereas if $x\in C_u$ we use \eqref{f:reprC}
and Theorem~\ref{t:sil} with $\Sigma = \partial^*\{u > \ut(x)\}$.

Let us prove \eqref{f:thetareg}.
Let $x\in \Omega \setminus (S_{\A} \cup S_u)$ be a point such that
\eqref{f:reprC} holds 
and 
\eqref{f:sil1}-\eqref{f:sil2} hold with $\Sigma = \partial^*\{u > \ut(x)\}$,
oriented in such a way that $\nu$ points inside the set $\{u > \ut(x)\}$
(observe that these properties hold for $|D^d u|$-a.e.\ point,
see the proof of Theorem~\ref{t:repr}).
Since, for every $r>0$,
\[
\frac{2N}{\omega_{N-1} r^N} \int_{B_r^i(x,\nu_u(x))} \frac{y-x}{|y-x|}\, dy =
\nu_u(x),
\qquad
\frac{2N}{\omega_{N-1} r^N} \int_{B_r^e(x,\nu_u(x))} \frac{y-x}{|y-x|}\, dy =
-\nu_u(x),
\]
then
\[
\begin{split}
\left|\Trp{\A}{\Sigma}(x) - \widetilde{\A}(x) \cdot \nu_u(x)\right|
& \leq {} 
\lim_{r\to 0}\left| \mean{B_r^i(x,\nu_u(x))} [\A(y) - \widetilde{\A}(x)] \cdot
\frac{2N (y-x)}{ |y-x|} \, dy \right|
\\ & \leq
\lim_{r\to 0} 2N
\mean{B_r^i(x,\nu_u(x))} |\A(y) - \widetilde{\A}(x)|\, dy = 0\,.
\end{split}
\]
An analogous computation shows that also
$\Trm{\A}{\Sigma}(x) = \widetilde{\A}(x)\cdot\nu_u(x)$, so that \eqref{f:thetareg} follows.
\end{proof}

\begin{remark}
In light of Remark \ref{rem:lambda_repr_case}, we may exploit Theorem \ref{t:sil} as in the proof of Corollary \ref{c:repr} in order to get
\begin{align*}
\theta_{\lambda}(\A, Du, x) & = (1 - \lambda(x)) \lim_{r\to 0}
\frac{N}{\omega_{N-1} r^N} \int_{B_r^i(x,\nu_u(x))} \A(y) \cdot \frac{y-x}{|y-x|}\, dy + \\
& - \lambda(x) \lim_{r\to 0}
\frac{N}{\omega_{N-1} r^N} \int_{B_r^e(x,\nu_u(x))} \A(y) \cdot \frac{y-x}{|y-x|}\, dy
\end{align*}
for $|D^j u|$-a.e. $x \in \Omega$.
\end{remark}

\section{Tangential properties of the pairing measure}

As a consequence of the representation formula in Theorem~\ref{t:repr}
we easily recover the local structure of the pairing measure by means of its
tangent measures.

For every $x\in \Omega$, let $\Ixr(y) := (y-x) / r$ denote the homothety with 
scaling 
factor $r$ mapping $x$ in $0$. For $r > 0$ small enough such that $B_r(x) \Subset \Omega$, 
the pushforward $\Ixr[x, r]_\#\mu$ of 
a Radon measure $\mu$ is the 
measure acting on a test function $\phi \in C_c(B_1)$ as 
\begin{equation*}
\int_{B_1} \phi \,d(\Ixr[x, r]_\#\mu)=\int_\Omega\phi\circ \Ixr d\mu.
\end{equation*}

\begin{definition}[Tangent measures]
\label{d:tang}
Let $\mu \in \mathcal{M}_{\rm loc}(\Omega)$.
We say that $\gamma$ is a tangent measure of $\mu$ at $x\in\Omega$
if $\gamma$ is a non-zero Radon measure and there exists some sequence $(r_i)$ satisfying $r_i\downarrow 0$ and such that
\[
\frac{1}{|\mu|(B_{r_i}(x)))}\, \Ixr[x, r_i]_\#\mu \overset{\ast}{\rightharpoonup} \gamma \text{ in } \mathcal{M}_{\rm loc}(B_1).
\]
We denote by $\Tan(\mu, x)$ the set of all tangent measures of $\mu$ at $x$.

For every $\alpha\geq 0$, we denote by $\Tan_\alpha(\mu, x)$ the family of
non-zero Radon measures $\gamma$ such that
there exists
a sequence $r_i\downarrow 0$ for which 
\[
r_i^{-\alpha}\,{\Ixr[x, r_i]_\#\mu}
\overset{\ast}{\rightharpoonup} \gamma \text{ in } \mathcal{M}_{\rm loc}(B_1).
\]
\end{definition}

Following the notation established in Section \ref{s:prelim}, in the following results for any given function $f \in L^{1}_{\rm loc}(\Omega, |\mu|)$ we use the notation $f(x) := \widetilde f(x)$ for every $x \in \Omega$ Lebesgue point of $f$ with respect to $|\mu|$. We start by proving the following property of the tangent measures (for related results see \cite[Theorem 2.44]{AFP} and \cite[Lemma~14.6]{Mattila}).

\begin{lemma}
\label{l:tang}
Let $\mu \in \mathcal{M}_{\rm loc}(\Omega)$, and let $f\in L^1_{\rm loc}(\Omega, |\mu|)$.
If $x\in\Omega$ is a Lebesgue point of $f$ with respect to $|\mu|$
and $f(x)\neq 0$, then
\[
\Tan(f\mu, x) = f(x)\, \Tan(\mu, x) \quad \text{ and } \quad \Tan_\alpha(f\mu, x) = f(x)\, \Tan_\alpha(\mu, x) \quad \forall \alpha \ge 0. \]
\end{lemma}

\begin{proof}
Let $x\in\Omega$ be as in the statement, and let
$r_i\downarrow 0$ (so that $B_{r_i}(x) \Subset \Omega$) be such that
at least one of the sequences
$c_i\, \Ixr[x,r_i]_\# (f\mu)$ and $c_i\, \Ixr[x,r_i]_\# \mu$
converges weakly${}^\ast$ to a Radon measure, where $c_i = \frac{1}{|\mu|(B_{r_i}(x)))}$.
To fix the ideas, assume that
$c_i\, \Ixr[x,r_i]_\# \mu \overset{\ast}{\rightharpoonup} \gamma$. 
For every $\varphi\in C_c(B_1)$, we have that
\[
c_i \left|
\int_{B_1} \varphi \, d \Ixr[x,r_i]_\# (f\mu)
- f(x) \int_{B_1} \varphi \, d \Ixr[x,r_i]_\# \mu
\right|
= 
\left|
\mean{B_{r_i}(x)} \varphi\left(\frac{y-x}{r_i}\right) [f(y) - f(x)] d\mu(y)
\right|\,,
\]
and the right-hand side converges to $0$ as $i\to +\infty$ 
since $x$ is a Lebesgue point of $f$ with respect to $\mu$.
It follows that the sequence
$c_i\, \Ixr[x,r_i]_\# (f\mu)$
converges weakly${}^\ast$ to $f(x)\, \gamma$, 
so that we can conclude that $\Tan(f\mu, x) = f(x)\, \Tan(\mu, x)$.

The equality for $\Tan_\alpha$ can be proved by dealing separately with the case in which both tangent sets are empty and the one in which they are not. In the latter we have $|\mu|(B_r(x)) \le C r^{\alpha}$ for some $C \ge 0$, so that one can argue as above.
\end{proof}

We can now state two results on the tangent measures of pairings. To this purpose, for any unit vector $\nu$ we set $\nu^\perp := \{ y \in \R^N : y \cdot \nu = 0 \}$.

\begin{theorem}
Let $\A\in\mathcal{DM}^{\infty}_{\rm loc}(\Omega)$ and
let $E$ be a set of locally finite perimeter in $\Omega$.
Let $x\in \partial^* E$ 
be a Lebesgue point of $\Trace[*]{\A}{\partial^* E}$
with respect to $|D \chi_E|$,
such that
$\Trace[*]{\A}{\partial^* E}(x) \neq 0$.
Then
\[
\Tan_{N-1}((\A, D\chi_E), x)
=
\Trace[*]{\A}{\partial^* E}(x)\,
\hh\res \nu_E^\perp(x).
\]
\end{theorem}

\begin{proof}
It is a consequence of Lemma~\ref{l:tang}, Corollary~\ref{c:paironch} and De Giorgi's theorem (see \cite[Lemma 3.58 and Theorem~3.59]{AFP}).
\end{proof}

\begin{theorem}
Let $\A\in\mathcal{DM}^{\infty}_{\rm loc}(\Omega)$ and let $u\in BV_{\rm loc}(\Omega) \cap L^\infty_{\rm loc}(\Omega)$. If $x\in \Omega \setminus S_u$ is a Lebesgue point of $\Trace[*]{\A}{\partial^*\{u > \ut(x)\}}$
with respect to $|D^du|$
such that
$\Trace[*]{\A}{\partial^*\{u > \ut(x)\}}(x) \neq 0$, then
\begin{equation*}
\Tan((\A, Du)^d , x) =
\Trace[*]{\A}{\partial^*\{u > \ut(x)\}}(x)\,
\Tan(|D^d u|, x).
\end{equation*}
If instead $x\in J_u$ 
is a Lebesgue point of $\Trace[*]{\A}{J_u}$
with respect to $|D^j u|$
such that
$\Trace[*]{\A}{J_u}(x) \neq 0$, then
\begin{equation*}
\Tan_{N-1}((\A, Du)^j, x) =
\Trace[*]{\A}{J_u}(x)\,
[u^+(x) - u^-(x)]\,
\hh\res \nu_u^\perp(x).
\end{equation*}
\end{theorem}

\begin{proof}
It is a consequence of Lemma~\ref{l:tang}, Theorem~\ref{t:repr} and the Federer-Vol'pert theorem (see \cite[Theorem~3.78]{AFP}).
\end{proof}

\begin{remark}
Exploiting Lemma~\ref{l:tang}, Remark \ref{rem:lambda_repr_case} and the Federer-Vol'pert theorem, we deduce the following representation for the tangent measure of the $\lambda$-pairing as well:
\begin{equation*}
\Tan_{N-1}((\A, Du)^{j}_{\lambda}, x) =
\Trace[\lambda]{\A}{J_u}(x) \,
[u^+(x) - u^-(x)]\,
\hh\res \nu_u^\perp(x),
\end{equation*}
for every $x\in J_u$ which is a Lebesgue point of $\Trace[\lambda]{\A}{J_u}$
with respect to $|D^j u|$
and such that
$\Trace[\lambda]{\A}{J_u}(x) \neq 0$,
where 
$$\Trace[\lambda]{\A}{J_u}(x) := (1-\lambda(x))\Trp{\A}{J_u}(x) +\lambda(x) \Trm{\A}{J_u}(x).$$
\end{remark}


\begin{bibdiv}
\begin{biblist}

\bib{AmbCriMan}{article}{
      author={Ambrosio, {L.}},
      author={Crippa, {G.}},
      author={Maniglia, {S.}},
       title={Traces and fine properties of a {$BD$} class of vector fields and
  applications},
        date={2005},
     journal={Ann. Fac. Sci. Toulouse Math. (6)},
      volume={14},
      number={4},
       pages={527\ndash 561},
         url={http://afst.cedram.org/item?id=AFST_2005_6_14_4_527_0},
}

\bib{ADM}{incollection}{
      author={Ambrosio, {L.}},
      author={De~Lellis, {C.}},
      author={Mal\'y, {J.}},
       title={On the chain rule for the divergence of {BV}-like vector fields:
  applications, partial results, open problems},
        date={2007},
   booktitle={Perspectives in nonlinear partial differential equations},
      series={Contemp. Math.},
      volume={446},
   publisher={Amer. Math. Soc., Providence, RI},
       pages={31\ndash 67},
         url={http://dx.doi.org/10.1090/conm/446/08625},
}

\bib{AFP}{book}{
      author={Ambrosio, {L.}},
      author={Fusco, {N.}},
      author={Pallara, {D.}},
       title={Functions of bounded variation and free discontinuity problems},
      series={Oxford Mathematical Monographs},
   publisher={The Clarendon Press Oxford University Press},
     address={New York},
        date={2000},
}

\bib{MR1814993}{article}{
   author={Andreu, F.},
   author={Ballester, C.},
   author={Caselles, V.},
   author={Maz\'{o}n, J. M.},
   title={The Dirichlet problem for the total variation flow},
   journal={J. Funct. Anal.},
   volume={180},
   date={2001},
   number={2},
   pages={347--403},
   doi={10.1006/jfan.2000.3698},
}
	

\bib{AVCM}{book}{
      author={Andreu-Vaillo, {F.}},
      author={Caselles, {V.}},
      author={Maz\'on, {J.M.}},
       title={Parabolic quasilinear equations minimizing linear growth
  functionals},
      series={Progress in Mathematics},
   publisher={Birkh\"auser Verlag, Basel},
        date={2004},
      volume={223},
         url={http://dx.doi.org/10.1007/978-3-0348-7928-6},
}

\bib{Anz}{article}{
      author={Anzellotti, {G.}},
       title={Pairings between measures and bounded functions and compensated
  compactness},
        date={1983},
     journal={Ann. Mat. Pura Appl. (4)},
      volume={135},
       pages={293\ndash 318 (1984)},
         url={http://dx.doi.org/10.1007/BF01781073},
}

\bib{Anz2}{misc}{
      author={Anzellotti, {G.}},
       title={Traces of bounded vector--fields and the divergence theorem},
        date={1983},
        note={Unpublished preprint},
}

\bib{BuffaComiMira}{article}{
   author={Buffa, V.},
   author={Comi, G. E.},
   author={Miranda, M. Jr.},
   title={On BV functions and essentially bounded divergence-measure fields
   in metric spaces},
   journal={Rev. Mat. Iberoam.},
   volume={38},
   date={2022},
   number={3},
   pages={883--946},
   doi={10.4171/rmi/1291},
}
	

\bib{Cas}{article}{
      author={Caselles, V.},
       title={On the entropy conditions for some flux limited diffusion
  equations},
        date={2011},
     journal={J. Differential Equations},
      volume={250},
      number={8},
       pages={3311\ndash 3348},
         url={http://dx.doi.org/10.1016/j.jde.2011.01.027},
}

\bib{ChCoTo}{article}{
      author={Chen, {G.-Q.}},
      author={Comi, {G. E.}},
      author={Torres, {M.}},
       title={Cauchy fluxes and {G}auss-{G}reen formulas for divergence-measure
  fields over general open sets},
        date={2019},
     journal={Arch. Ration. Mech. Anal.},
      volume={233},
      number={1},
       pages={87\ndash 166},
         url={https://doi.org/10.1007/s00205-018-01355-4},
}

\bib{ChenFrid}{article}{
      author={Chen, {G.-Q.}},
      author={Frid, {H.}},
       title={Divergence-measure fields and hyperbolic conservation laws},
        date={1999},
     journal={Arch. Ration. Mech. Anal.},
      volume={147},
      number={2},
       pages={89\ndash 118},
         url={http://dx.doi.org/10.1007/s002050050146},
}

\bib{ChFr1}{article}{
      author={Chen, {G.-Q.}},
      author={Frid, {H.}},
       title={Extended divergence-measure fields and the {E}uler equations for
  gas dynamics},
        date={2003},
     journal={Comm. Math. Phys.},
      volume={236},
      number={2},
       pages={251\ndash 280},
         url={http://dx.doi.org/10.1007/s00220-003-0823-7},
}


\bib{ChTo}{article}{
      author={Chen, {G.-Q.}},
      author={Torres, {M.}},
       title={On the structure of solutions of nonlinear hyperbolic systems of
  conservation laws},
        date={2011},
     journal={Commun. Pure Appl. Anal.},
      volume={10},
      number={4},
       pages={1011\ndash 1036},
         url={http://dx.doi.org/10.3934/cpaa.2011.10.1011},
}

\bib{ComiMagna}{article}{
   author={Comi, G. E.},
   author={Magnani, V.},
   title={The Gauss-Green theorem in stratified groups},
   journal={Adv. Math.},
   volume={360},
   date={2020},
   pages={106916, 85},
   doi={10.1016/j.aim.2019.106916},
}

\bib{ComiPayne}{article}{
   author={Comi, G. E.},
   author={Payne, K. R.},
   title={On locally essentially bounded divergence measure fields and sets
   of locally finite perimeter},
   journal={Adv. Calc. Var.},
   volume={13},
   date={2020},
   number={2},
   pages={179--217},
   doi={10.1515/acv-2017-0001},
}

\bib{comi2023fractional}{article}{
  title={Fractional divergence-measure fields, Leibniz rule and Gauss--Green formula},
  author={Comi, G. E.},
  author={Stefani, G.},
  journal={Bollettino dell'Unione Matematica Italiana},
  pages={1--23},
  year={2023},
  publisher={Springer}
}


\bib{CD3}{article}{
      author={Crasta, {G.}},
      author={De~Cicco, {V.}},
       title={Anzellotti's pairing theory and the {G}auss--{G}reen theorem},
        date={2019},
     journal={Adv. Math.},
      volume={343},
       pages={935\ndash 970},
         url={https://doi.org/10.1016/j.aim.2018.12.007},
}

\bib{CD4}{article}{
  title={An extension of the pairing theory between divergence-measure fields and BV functions},
  author={Crasta, G.},
  author={De Cicco, V.},
  journal={J. Funct. Anal.},
  volume={276},
  number={8},
  pages={2605--2635},
  year={2019},
  publisher={Elsevier}
}

\bib{CD5}{article}{
author={Crasta, {G.}},
      author={De~Cicco, {V.}},
       title={On the variational nature of the Anzellotti pairing},
note={ArXiv:2207.06469},
url={https://doi.org/10.48550/arXiv.2302.10592}
}

\bib{CDM}{article}{
      author={Crasta, {G.}},
      author={De~Cicco, {V.}},
      author={Malusa, {A.}},
       title={Pairings between bounded divergence-measure vector fields and
  {BV} functions},
        date={2022},
     journal={Adv. Calc. Var.},
       pages={787-810},
             volume={15},
            number={4},
         url={https://doi.org/10.1515/acv-2020-0058},
}

\bib{MR3939259}{article}{
   author={De Cicco, V.},
   author={Giachetti, D.},
   author={Segura de Le\'{o}n, S.},
   title={Elliptic problems involving the 1-Laplacian and a singular lower
   order term},
   journal={J. Lond. Math. Soc. (2)},
   volume={99},
   date={2019},
   number={2},
   pages={349--376},
   doi={10.1112/jlms.12172},
}


\bib{DGMM}{article}{
      author={Degiovanni, {M.}},
      author={Marzocchi, {A.}},
      author={Musesti, {A.}},
       title={Cauchy fluxes associated with tensor fields having divergence
  measure},
        date={1999},
     journal={Arch. Ration. Mech. Anal.},
      volume={147},
      number={3},
       pages={197\ndash 223},
         url={http://dx.doi.org/10.1007/s002050050149},
}

\bib{Falconer}{book}{
      author={Falconer, {K.}},
       title={Fractal geometry},
     edition={Third},
   publisher={John Wiley \& Sons, Ltd., Chichester},
        date={2014},
        note={Mathematical foundations and applications},
}

\bib{FED}{book}{
      author={Federer, H.},
       title={Geometric measure theory},
      series={Die Grundlehren der mathematischen Wissenschaften, Band 153},
   publisher={Springer-Verlag New York Inc., New York},
        date={1969},
}

\bib{GMS1}{book}{
      author={Giaquinta, {M.}},
      author={Modica, {G.}},
      author={Sou{\v{c}}ek, {J.}},
       title={Cartesian currents in the calculus of variations. {I}},
      series={Ergebnisse der Mathematik und ihrer Grenzgebiete. 3. Folge. A
  Series of Modern Surveys in Mathematics [Results in Mathematics and Related
  Areas. 3rd Series. A Series of Modern Surveys in Mathematics]},
   publisher={Springer-Verlag},
     address={Berlin},
        date={1998},
      volume={37},
        note={Cartesian currents},
}


\bib{MR2348842}{article}{
   author={Kawohl, B.},
   author={Schuricht, F.},
   title={Dirichlet problems for the 1-Laplace operator, including the
   eigenvalue problem},
   journal={Commun. Contemp. Math.},
   volume={9},
   date={2007},
   number={4},
   pages={515--543},
   doi={10.1142/S0219199707002514},
}

\bib{ComiLeo}{article}{
author={Leonardi, G. P.},
author = {Comi, G. E.},
title={The prescribed mean curvature equation with measure data},
note={ArXiv:2302.10592},
url={https://doi.org/10.48550/arXiv.2207.06469}
}

\bib{LeoSar}{article}{
      author={Leonardi, G. P.},
      author={Saracco, G.},
       title={The prescribed mean curvature equation in weakly regular
  domains},
        date={2018},
     journal={NoDEA Nonlinear Differential Equations Appl.},
      volume={25},
      number={2},
       pages={Art. 9, 29},
         url={https://doi.org/10.1007/s00030-018-0500-3},
}

\bib{LeoSar2}{article}{
   author={Leonardi, G. P.},
   author={Saracco, G.},
   title={Rigidity and trace properties of divergence-measure vector fields},
   journal={Adv. Calc. Var.},
   volume={15},
   date={2022},
   number={1},
   pages={133--149},
   doi={10.1515/acv-2019-0094},
}

\bib{Mattila}{book}{
      author={Mattila, {P.}},
       title={Geometry of sets and measures in {E}uclidean spaces},
      series={Cambridge Studies in Advanced Mathematics},
   publisher={Cambridge University Press, Cambridge},
        date={1995},
      volume={44},
         url={https://doi.org/10.1017/CBO9780511623813},
        note={Fractals and rectifiability},
}

\bib{MR2502520}{article}{
   author={Mercaldo, A.},
   author={Segura de Le\'{o}n, S.},
   author={Trombetti, C.},
   title={On the solutions to 1-Laplacian equation with $L^1$ data},
   journal={J. Funct. Anal.},
   volume={256},
   date={2009},
   number={8},
   pages={2387--2416},
   doi={10.1016/j.jfa.2008.12.025},
}


\bib{MR3501836}{article}{
   author={Scheven, C.},
   author={Schmidt, T.},
   title={BV supersolutions to equations of 1-Laplace and minimal surface
   type},
   journal={J. Differential Equations},
   volume={261},
   date={2016},
   number={3},
   pages={1904--1932},
   doi={10.1016/j.jde.2016.04.015},
}

\bib{MR3813962}{article}{
   author={Scheven, C.},
   author={Schmidt, T.},
   title={On the dual formulation of obstacle problems for the total
   variation and the area functional},
   journal={Ann. Inst. H. Poincar\'{e} C Anal. Non Lin\'{e}aire},
   volume={35},
   date={2018},
   number={5},
   pages={1175--1207},
   doi={10.1016/j.anihpc.2017.10.003},
}
	

\bib{Schu}{article}{
      author={Schuricht, {F.}},
       title={A new mathematical foundation for contact interactions in
  continuum physics},
        date={2007},
     journal={Arch. Ration. Mech. Anal.},
      volume={184},
      number={3},
       pages={495\ndash 551},
         url={http://dx.doi.org/10.1007/s00205-006-0032-6},
}

\bib{Silh}{article}{
      author={\v{S}ilhav\'{y}, {M.}},
       title={Divergence measure fields and {C}auchy's stress theorem},
        date={2005},
     journal={Rend. Sem. Mat. Univ. Padova},
      volume={113},
       pages={15\ndash 45},
}

\bib{Silh2019}{article}{
      author={\v{S}ilhav\'{y}, {M.}},
       title={The Gauss-Green theorem for bounded vectorfields with
  divergence measure on sets of finite perimeter},
     journal={Indiana Univ. Math. J.},
   volume={72},
   date={2023},
   number={1},
   pages={29--42},
}

\end{biblist}
\end{bibdiv}

\def\cprime{$'$}

\end{document}